\documentclass{article}
\usepackage[utf8]{inputenc}
\usepackage{todonotes}
\usepackage{hyperref}
\usepackage{bbm}
\usepackage{amsfonts}
\usepackage{amsmath}
\usepackage{amssymb}
\usepackage{mathtools}
\usepackage{amsthm}
\usepackage{cleveref}
\usepackage{xfrac}
\usepackage{float}
\usepackage{svg}
\usepackage[title]{appendix}
\usepackage[a4paper, margin=1in]{geometry}
\usepackage{xspace}
\usepackage{multirow}
\usepackage{comment}
\usepackage{stackengine}
\usetikzlibrary {patterns,patterns.meta} 
\usetikzlibrary{fit}
\usetikzlibrary{arrows, arrows.meta}

\title{How Balanced Can Permutations Be?}
\author{Gal Beniamini\qquad Nir Lavee\qquad Nati Linial
\thanks{Supported in part by a NSF-BSF research grant ``Global Geometry of Graphs''.}
\vspace {0.1cm} \\\small The Hebrew University of Jerusalem}
\date{}

\hypersetup{colorlinks=true}	

\newtheorem{theorem}{Theorem}[section]
\newtheorem*{theorem*}{Theorem}
\newtheorem{corollary}{Corollary}[theorem]
\newtheorem{lemma}[theorem]{Lemma}
\newtheorem{proposition}[theorem]{Proposition}

\newtheorem{definition}[theorem]{Definition}
\newtheorem{notation}[theorem]{Notation}

\newcommand{\twob}{$2$-balanced\xspace}
\newcommand{\threeb}{$3$-balanced\xspace}

\newcommand{\Sn}{\mathbb{S}_n}

\newcommand{\RR}{\mathbb{R}}

\newcommand{\E}{\mathbb{E}}

\newcommand{\pc}[2]{{\# \mathtt{ #1 } \left( #2 \right)}}

\newcommand{\eqdef}{\vcentcolon=}

\DeclareMathOperator{\ES}{ES}

\newcommand{\ESZ}{Erd\H{o}s-Szekeres\xspace}

\theoremstyle{remark}
\newtheorem{rem}[theorem]{\protect\remarkname}
\providecommand{\remarkname}{Remark}

\newcommand{\perm}[1]{(\mathtt{ #1 })}

\def\acts{\curvearrowright}

\pgfdeclarepattern{
  name=hatch,
  parameters={\hatchsize,\hatchangle,\hatchlinewidth},
  bottom left={\pgfpoint{-.1pt}{-.1pt}},
  top right={\pgfpoint{\hatchsize+.1pt}{\hatchsize+.1pt}},
  tile size={\pgfpoint{\hatchsize}{\hatchsize}},
  tile transformation={\pgftransformrotate{\hatchangle}},
  code={
    \pgfsetlinewidth{\hatchlinewidth}
    \pgfpathmoveto{\pgfpoint{-.1pt}{-.1pt}}
    \pgfpathlineto{\pgfpoint{\hatchsize+.1pt}{\hatchsize+.1pt}}
    \pgfpathmoveto{\pgfpoint{-.1pt}{\hatchsize+.1pt}}
    \pgfpathlineto{\pgfpoint{\hatchsize+.1pt}{-.1pt}}
    \pgfusepath{stroke}
  }
}
\tikzset{
  hatch size/.store in=\hatchsize,
  hatch angle/.store in=\hatchangle,
  hatch line width/.store in=\hatchlinewidth,
  hatch size=5pt,
  hatch angle=0pt,
  hatch line width=.5pt,
}

\theoremstyle{plain}
\newtheorem{thm}{Theorem}

\begin{document}

\maketitle

\vspace{-0.45cm}
\begin{abstract}
A permutation $\pi \in \Sn$ is {\em $k$-balanced} if every permutation of order $k$ occurs in $\pi$ equally often, through order-isomorphism.
In this paper, we explicitly construct $k$-balanced permutations for $k \le 3$, and every $n$
that satisfies the necessary divisibility conditions.
In contrast, we prove that for $k \ge 4$, no such permutations exist.  
In fact, we show that in the case $k \ge 4$, every $n$-element permutation is at least $\Omega_n(n^{k-1})$
far from being $k$-balanced.
This lower bound is matched for $k=4$, by a construction based on the \ESZ permutation.

\end{abstract}

\section{Introduction}
\label{section:introduction}
A permutation $\tau \in \mathbb{S}_k$ occurs as a \textit{pattern} in
$\pi \in \Sn$, if there are indices $1\le s_1 < \dots < s_k\le n$
such that $$\forall i,j \in \{1, \dots, k\}:\ \pi(s_i) < \pi(s_j) \iff \tau(i) < \tau(j).$$
This simple concept of \textit{order-isomorphism} gives rise to many intriguing problems.

\paragraph{The local structure of permutations.}
The $k$-profile of an $n$-element permutation $\pi$
is the vector that counts the occurrences
of every order-$k$ pattern in $\pi$. There are still many things we do not know about
the set of $k!$-dimensional vectors that arise in this way.
The \ESZ Theorem \cite{erdos1935combinatorial} 
states that every order-$n$
permutation must contain a monotone pattern of order
$\lceil\sqrt{n}~\rceil$.
Conversely, the \textit{packing density} of a pattern $\tau \in \mathbb{S}_k$
is its maximal proportion within the $k$-profile of an $n$-element
permutation, as $n \to \infty$. 
Packing densities have received considerable attention \cite{albert2002packing,sliacan2018improving,wilf2002patterns},
yet even for $\mathbb{S}_4$ some answers remain presently unknown. 
Patterns in random permutations have also received their share of attention \cite{even2020patterns, janson2013asymptotic}, as has the algorithmic problem of computing the $k$-profile \cite{even2021counting,dudek2020counting}. Both are relevant to certain basic questions in mathematical statistics. Every one of the above questions also has a graph-theoretic analogue.
For instance, the counterpart to \ESZ' Theorem is Ramsey's Theorem, and the graph-theoretic equivalent to packing density is \emph{inducibility} \cite{pippenger1975inducibility}, and so forth.

\subsection{Our Contribution}

In this paper we consider the permutation-theoretic analogue of \emph{combinatorial designs}.
For integers $n \ge k \ge 1$, we say that a permutation
$\pi \in \Sn$ is \emph{$k$-balanced}
if every order-$k$ pattern occurs in $\pi$ equally
often, i.e. exactly $\binom{n}{k}/k!$ times.
By way of example, $\pi = \mathtt{2413}$ is $2$-balanced. 
So, for which values of $n$ and $k$ does there exist a $k$-balanced permutation $\pi\in\mathbb{S}_{n}$? We answer this question fully.

\paragraph{Constructions for $k \le 3$.}
It is not hard to see that any $k$-balanced permutation is $(k-1)$-balanced.
Therefore, for an $n$-element permutation to be $k$-balanced, we must at least have $r! \mid \binom{n}{r}$, for every $r \le k$.
It is a straightforward exercise to see that these necessary divisibility conditions suffice in the case $k=2$,
that is, $2$-balanced permutations exist for every \textit{admissible} $n$. 
Is the same true for $k=3$?
Our first result answers this in the positive, resolving an open question of \cite{cooper2008symmetric}.

\begin{thm}
    \label{thm_intro_3balanced}
    For $k\leq 3$ and every $n$, there exists a $k$-balanced permutation in $\Sn$ iff $n$ is admissible.
\end{thm} 

Our construction for $k=3$ is explicit and relies heavily on \emph{rotation-invariance} (see \Cref{section:3balanced_construction}).
The divisibility conditions for $3$-balanced permutations permit six remainders in $\mathbb{Z}/36\mathbb{Z}$,
and we provide infinite families corresponding to each.
These families are all based on a single basic construction, and together cover all
but a spurious collection of $19$ admissible values of $n$, which we handle individually.

\Cref{thm_intro_3balanced} proves the existence of $3$-balanced ``designs''.
This has many noteworthy counterparts.
For example, a \emph{$(r,s; n)$-Steiner system} is a collection
 of $r$-element subsets of $\{1, 2, \dots, n\}$, such that every
 $s$-element subset is contained in the same number of members in the system.
 Such a system can exist only if $n, r$ and $s$ satisfy certain arithmetic conditions,
 and a major discovery of Keevash \cite{keevash2014existence} (see also \cite{glock2016existence}) says that
 for fixed $r>s>1$ and for large enough $n$, if the above arithmetic conditions
 are satisfied, then a Steiner system exists.
 Similarly, in graph theory, Janson and Spencer \cite{janson1992probabilistic} considered \textit{proportional graphs},
 in which every subgraph of a fixed size appears the exact number of times as expected. 
 They showed that with respect to order-$3$ subgraphs, there exist infinitely many proportional graphs.
 Finally, with regards to permutations,
 Cooper and Petrarca \cite{cooper2008symmetric} noted the existence of $3$-balanced permutations for $n=9$ (the least admissible $n$), and our \Cref{thm_intro_3balanced} extends this to \textit{every} admissible $n$.

\paragraph{Non-existence for $k \ge 4$.} For our second result, we prove that there are no $4$-balanced permutations. As the $k$-balanced condition is downwards closed in $k$, we obtain the following.

\begin{thm}
    \label{thm_no_4_balanced}
    There are no $k$-balanced permutations for $k\geq 4$.
\end{thm}

The proof of \Cref{thm_no_4_balanced} follows by establishing a simple polynomial identity relating entries of the $r$-profile of any permutation, for $r \le 4$ (see \Cref{section:4balanced_nonex}). The $4$-balanced profile violates this
identity. This resolves another open question of \cite{cooper2008symmetric}, who carried out a large (non-exhaustive) computer search for $n=64$, the smallest admissible $4$-balanced cardinality. This explains why none were found.

\Cref{thm_no_4_balanced} is closely related to a result of
Naves, Pikhurko and Scott \cite{naves2018unproportional}, who proved the non-existence of proportional graphs,
in which every order-$4$ subgraph appears exactly the expected number of times. 
Their proof similarly relies on a 
polynomial identity. Our result is also related to quasirandom permutations,
i.e.,  infinite families in which the normalised $k$-profile converges \textit{asymptotically} to uniform, as $n$ tends to infinity.
The theory of graph limits and graphons \cite{lovasz2012large} has been highly influential 
in graph theory in recent years, and an analogous theory concerning limits of
permutations and the notion of permutons has been investigated as well,
e.g., \cite{cooper2004quasirandom,hoppen2013limits, hoppen2011limits}.
Notions of pseudo-random graphs were introduced by Thomason \cite{thomason1987pseudo}
and a remarkable result of Chung, Graham and Wilson \cite{chung1989quasi} shows that a graph
is pseudo-random iff it has the ``right'' number of $4$-cycles. Proving
a conjecture of R. Graham (see \cite{cooper2004quasirandom}), Kr\`al' and Pikhurko \cite{kral2013quasirandom}
proved that a permuton is quasirandom iff it is $4$-symmetric. Our techniques in
proving \Cref{thm_no_4_balanced} differ from those of \cite{kral2013quasirandom}, 
and we point out the difficulty in applying the latter to 
the discrete setting in \Cref{subsect:comparison_kp}.

\paragraph{Minimum distance from $k$-balanced.}
If (as we show)
$k$-balanced permutations do not exist for $k \ge 4$, 
how \textit{close} to balanced can they be?
Formally, we define the \textit{distance} of $\pi \in \Sn$ from being $k$-balanced,
to be the $\ell_\infty$-distance between $\pi$'s $k$-profile and the uniform vector $(\binom{n}{k}/k!)\mathbbm{1}$.
We prove:

\begin{thm}
    \label{thm_lb_distance}
    For $k\geq 4$, the distance of every $n$-element permutation from being $k$-balanced is $\Omega_n(n^{k-1})$.
\end{thm}

Our proof of \Cref{thm_lb_distance} can be viewed as a robust version of
our proof of \Cref{thm_no_4_balanced}, using the same
polynomial identity (see \Cref{section:min_dist_kbal}).  
We prove that the bound in \Cref{thm_lb_distance} is tight for $k = 4$.
That is, we give a construction of permutations that attain this distance.
This construction is based on a modification of the well-known \ESZ permutation
\cite{erdos1935combinatorial}.
For larger $k$ the tightness of our bound remains open.
However, we note that \textit{all} entries
in the $k$-profile of a uniformly random permutation in $\mathbb{S}_n$ 
are, with probability $>99\%$ (for large enough $n$), within distance
$\Theta_n(n^{k - 1/2})$ from $\binom{n}{k}/k!$ (see \Cref{subsubsect:random_perms}).
So, in the remaining cases, our bound is at most $\mathcal{O}_n\left(\sqrt{n}\right)$-far from tight. 

\paragraph{Relation between profiles and permutations.} Our last result is of a slightly different flavour,
and is of interest only when $k=k(n)$ grows with $n$. Given a $k$-profile we seek properties which are
common to all $n$-element permutations that have this profile.

\begin{thm}
    \label{thm_perm_to_points}
    There exists a set of $\widetilde{\Omega}(k^2/n)$ points in the $[n] \times [n]$ grid, 
    such that any two $n$-element permutations with the same $k$-profile, coincide in their restriction to this set.
\end{thm}

Our proof of \Cref{thm_perm_to_points} is established by drawing a connection
between polynomials and $k$-profiles (see \Cref{section:profiles_to_perms}).
We introduce a notion of evaluating a bivariate polynomial on a permutation,
and show that fixing the $k$-profile
of a permutation, also fixes the evaluation of all bivariate polynomials of degree $<k$ on it. 
Using results from approximation theory, this allows us to construct a  
family of low-degree polynomials, whose evaluations express permutation points in terms of the $k$-profile alone. 

\paragraph{Open questions} There remain many interesting questions. For instance, is the distance lower bound tight? And for how many patterns simultaneously? We refer the reader to the discussion in \Cref{sect:discussion}.
\section{Preliminaries}

As usual, we denote the symmetric group of order $n$ by $\mathbb{S}_n$.
By default we write permutations in $\mathbb{S}_n$ in the \emph{one-line notation}, and
think of a permutation as a bijection from $[n]$ to itself,
where  $[n] \eqdef \{1, 2, \dots, n\}$.  Any finite set of points in the plane, no two of which are
axis-aligned, defines a permutation.
Let $\mathcal{A} = \{ (x_1, y_1), (x_2, y_2), \dots, (x_n, y_n) \} \subset \RR^2$ be a set of points,
where $x_1<\ldots<x_n$ and all $y_i$ are distinct. 
Then, corresponding to $\mathcal{A}$ is the permutation $\sigma \in \Sn$ (denoted $\mathcal{A} \cong \sigma$), where $y_{\sigma^{-1}(1)} < y_{\sigma^{-1}(2)} < \dots < y_{\sigma^{-1}(n)}$.

The \emph{order-isomorphism} of permutations is in the focus of our work.

\begin{definition}[order-isomorphism]
    \label{defn:order_isom}
    Let $\pi \in \Sn$ and $\tau \in \mathbb{S}_k$ be permutations, where $k \le n$. Let $S = \{s_1, \dots, s_k\} \subseteq [n]$, where $s_1 < \dots < s_k$. We say that $\pi$ induced on $S$ is order-isomorphic to $\tau$ if:
    \[
        \forall i,j \in [k]:\ \pi(s_i) < \pi(s_j) \iff \tau(i) < \tau(j)
    \]    
and we denote this condition by $\pi(S) \cong \tau$.
When this is the case, we say the pattern $\tau$ occurs in $\pi$. 
The number of occurrences of $\tau$ in $\pi$ is denoted $\pc{\tau}{\pi}$, where:
\[
    \pc{\tau}{\pi} \eqdef \Big| \Big\{ S \in \binom{[n]}{k}\ :\ \pi(S) \cong \tau \Big\} \Big|
\]
\end{definition}

Thus, e.g., $\pc{123}{\pi}$ indicates the number of ascending triples in $\pi$. 
We also define:

\begin{definition}[$k$-profile of a permutation]
    \label{defn:perm_profile}
    Let $\pi \in \Sn$ be a permutation and let $1 \le k \le n$ be an integer. The $k$-profile of $\pi$ is defined as follows:
    \[
        \mathcal{P}_k(\pi) \eqdef \big( \pc{\tau}{\pi} \big)_{\tau \in \mathbb{S}_k} \in \RR_{\ge 0}^{\mathbb{S}_k}
    \]
This is a vector of $|\mathbb{S}_k|=k!$ non-negative integers that sum to $\binom{n}{k}$.
\end{definition}

This brings us to our main object of study.

\begin{definition}[$k$-balanced permutation] 
We say that a permutation $\pi \in \Sn$ is $k$-balanced for some $1 \leq k \le n$ if:
    \[
        \forall \tau \in \mathbb{S}_k:\ \pc{\tau}{\pi} = \frac{\binom{n}{k}}{k!}
    \]
\end{definition}

\subsection{Basic Observations on Balanced Permutations}

We first observe that
the $k$-profile of a permutation uniquely determines its $r$-profile
for every $r<k$, and in particular
every $k$-balanced permutation is also $r$-balanced.  

\begin{proposition}[Downward induction of pattern distribution]
\label{prop:down_ind_k_prof}
Let $n>k>r$ be positive integers. If $\pi \in \Sn$ and $\tau \in \mathbb{S}_{r}$, then
\begin{equation}\label{eq:induce}
        \binom{n-r}{k-r} \cdot \pc{\tau}{\pi} = \sum_{\sigma \in \mathbb{S}_k} \pc{\tau}{\sigma} \cdot \pc{\sigma}{\pi}
\end{equation}
\end{proposition}
\begin{proof}
Consider the pairs $B\subset A\subset [n]$ where $|B|=r$, $|A|=k$ and $\pi(B) \cong \tau$.
The r.h.s.\ expression is obtained by
grouping together in this count all sets $A$ with $\pi(A) \cong \sigma$,
for every permutation $\sigma\in \mathbb{S}_{k}$.
For the l.h.s.,\ note that for every $B\subset [n]$ with $\pi(B) \cong \tau$,
there are exactly $\binom{n-r}{k-r}$ sets $A$ with $B\subset A\subset [n]$.
\end{proof}

\begin{corollary}[$k$-balanced implies $(<k)$-balanced]
    \label{cor:k_bal_implies_lt_k_bal}
For $n>k>r$, every $k$-balanced permutation $\pi \in \Sn$ is also $r$-balanced.
\end{corollary}
\begin{proof}
By a simple inductive argument it suffices to consider the case $r=k-1$. So,
fix some $\tau\in \mathbb{S}_{k-1}$ and notice that the sum 
$\sum_{\sigma \in \mathbb{S}_k} \pc{\tau}{\sigma}$
has nothing to do with $\pi$. To evaluate this sum, think of $\tau$
as an ordering of the integers in $[k-1]$, and view an occurrence of $\tau$
in some $\sigma \in \mathbb{S}_k$ as that ordering in which exactly one of the numbers
$\{j-\frac 12 : j=1,\ldots,k\}$ is inserted somewhere. This perspective yields 
$\sum_{\sigma \in \mathbb{S}_k} \pc{\tau}{\sigma}=k^2$, as there are $k$ choices
for $j$ as above and $k$ spots at which $j-\frac 12$ can be inserted.
If $\pi$ is $k$-balanced, then $\pc{\sigma}{\pi}=\binom{n}{k}/{k!}$,
and as claimed, $\pc{\tau}{\pi}=\binom{n}{k-1}/{(k-1)!}$ by
 Equation (\ref{eq:induce}).
\end{proof}

\Cref{cor:k_bal_implies_lt_k_bal}, yields the following \emph{divisibility conditions}.

\begin{corollary}[divisibility conditions for $k$-balanced permutations]
    \label{cor:div_conds}
    If $\pi \in \Sn$ is $k$-balanced for some $1 \leq k \le n$, then $r! \mid \binom{n}{r}$ for all $1 \leq r \le k$.
\end{corollary}

\section{\texorpdfstring{$k$}{k}-Balanced Permutations for \texorpdfstring{$k\le3$}{k<=3}}
\label{section:3balanced_construction}

In this section we show that for $k=2$ and $k=3$, the  divisibility conditions 
of \Cref{cor:div_conds} are not only necessary, but also \emph{sufficient}.
Namely, we show that a $k$-balanced permutation on $n$ elements 
exists, whenever $n$ satisfies that arithmetic condition.
As a warmup, we first describe the case $k=2$.

\subsection{\texorpdfstring{$2$}{2}-Balanced Family}
\label{subsect:2bal_construction}

The following is one way (of many) to construct such an infinite family.

\begin{proposition}[$2$-balanced family]
There exists a \twob permutation in $\Sn$ if and only if 
\mbox{$n\equiv 0 \pmod 4$} or $n\equiv 1 \pmod 4$.
\end{proposition}

\begin{proof}
For any permutation $\pi \in \Sn$, swapping a pair of \emph{adjacent} elements in $\pi$ only changes that pair from ascending to descending or vice versa, leaving all other pairs unaffected. So, this increments one of $\pc{12}{\pi}$, $\pc{21}{\pi}$ by one, and decrements the other.

The identity permutation $\pi=(1,2,\dots, n) \in \Sn$ clearly satisfies
$\pc{12}{\pi}=\binom{n}{2}$ and $\pc{21}{\pi}=0$. At the other extreme,
the descending permutation $\tau = (n, \dots, 2, 1) \in \Sn$
satisfies $\pc{21}{\tau}=\binom{n}{2}$ and $\pc{12}{\tau}=0$.
It is possible to move from $\pi$ to $\tau$ by
a sequence of adjacent swaps (e.g., by ``bubble-sorting''). 
By a simple continuity argument, if $\binom{n}{2}$ is even, then there exists an intermediate permutation $\sigma$ satisfying $\pc{12}{\sigma} = \pc{21}{\sigma} = \binom{n}{2} / 2$. Finally, $\binom{n}{2}$ is even exactly when $n$ is $0$ or $1 \pmod 4$.
\end{proof}

The case $k=2$ is deceptively simple: the plot thickens for larger $k$, and
it is far from the truth that 
any collection of $k!$ non-negative integers summing up to $\binom{n}{k}$ is a $k$-profile of some permutation $\pi \in \Sn$. E.g, by the \ESZ Theorem \cite{erdos1935combinatorial}, we cannot have $\pc{123}{\pi}=\pc{321}{\pi}=0$ whenever $n\geq 5$. 

\subsection{\texorpdfstring{$3$}{3}-Balanced Family}
\label{subsect:3bal_construction}

To construct a \threeb family, we take a different approach. Consider the action $D_4 \acts \Sn$ of the dihedral group on $\Sn$, where we view
any permutation $\pi \in \Sn$ as the set of points $\{(i, \pi(i))\}$ in $\RR^2$, and act on the square $[1,n]^2$ in the standard way. This group action has the useful property that it respects pattern counts, in the following sense:

\begin{lemma}[pattern counts under $D_4 \acts \Sn$]
    \label{lemma:pc_D4}
    Let $\pi\in\Sn$ and $\tau\in\mathbb{S}_k$ be two permutations, and let $g \in D_4$. Then,  $\pc{\tau}{\pi} = \pc{\mathnormal{g.\tau}}{g.\pi}$.
\end{lemma}

\begin{proof}
    Any occurrence of $\tau$ in $\pi$ is associated with a set of points in $[1,n]^2$ where $\tau$ appears. These points are rigidly mapped by $g$ to points at which $g.\tau$ appears, therefore  $\pc{\tau}{\pi} = \pc{\mathnormal{g.\tau}}{g.\pi}$.
\end{proof}

Consequently, if the permutation $\pi \in \mathbb{S}_n$ is $3$-balanced, then so are
all the permutations in $\pi$'s orbit under the action of $D_4$. This orbit may include at most $|D_4| = 8$ permutations. Conversely, for $n > 1$, the orbit \emph{must} include at least two permutations, since no permutation is identical to its reflections about the horizontal and vertical axes, respectively (they agree on no more than one point).

For any element $g \in D_4$ and permutation $\pi \in \Sn$, we say that $\pi$ is \emph{$g$-invariant} if $g.\pi = \pi$ (i.e., $g$ is in the stabilizer of $\pi$). As noted, clearly no permutation is invariant to the involutions of the reflections about either axis. However, \emph{rotation-invariant} permutations do exist. In other words, $r.\pi = \pi$ where $r \in D_4$ is the $90^\circ$-rotation of the square (therefore, $g.\pi = \pi$ for all $g\in \langle r \rangle$). Rotation-invariant permutations are simply characterised, as follows.

\begin{proposition}[characterisation of rotation-invariant permutations]
    \label{prop:char_rotation_inv}
    Let $n>1$ be even.\footnote{For odd $n$, any rotation-invariant permutation \emph{must} include the point at the centre. The permutation induced on the remaining rows and columns is rotation-invariant, to which \Cref{prop:char_rotation_inv} now applies.} Then,
    \begin{enumerate}
        \item There exists a rotation-invariant $\pi \in \Sn$ if and only if $n=4m$, for some natural $m$. 
        \item Let $A \sqcup B = [2m]$ with $|A|=|B|=m$, and let $\sigma:A \to B$ be a bijection between $A$ and $B$. To every such $A, B$ and $\sigma$ there corresponds a rotation-invariant permutation in $\mathbb{S}_{4m}$. All rotation-invariant permutations in $\mathbb{S}_{4m}$ are generated in this way. 
    \end{enumerate}
\end{proposition}
\begin{proof}
    Consider the action $\langle r \rangle \acts [4m]^2$ of $r \in D_4$, the $90^{\circ}$-clockwise rotation. The orbit of a point $(x,y) \in [4m]^2$:
    \[
        O(x,y) = \{(x, y), (y, 4m-x+1), (4m-x+1, 4m-y+1), (4m-y+1, x)\}
    \]

    Let $A, B$ and $\sigma$ be as specified.
    For every $i \in A$ consider the orbit $O(i, \sigma(i))$ and its projections on the coordinate axes. Note that the four integers
$\{i, \sigma(i), 4m-i+1, 4m-\sigma(i)+1\}$ are all distinct. 
Consequently, \emph{all} $2m$ orbits, $\sqcup_{i \in A} O(i, \sigma(i)) \subseteq [4m]^2$ comprise a set of cardinality $4m$ with no repeated coordinates, and therefore define a permutation in $\mathbb{S}_{4m}$.
    
    Conversely, suppose $\pi\in\mathbb{S}_n$ is rotation-invariant where $n$ is even. The action $\langle r \rangle \acts [n]^2$ maps quadrants to quadrants, therefore they must each contain exactly $n/4$ points, and $n = 4m$ for some natural $m$. Let $A := \{ x : x,\pi(x) \in [2m] \}$ and $B := \{ y : \pi^{-1}(y),y \in [2m] \}$. By the previous argument, we have $\{ (i, \pi(i)) : i \in [4m] \} = \sqcup_{i \in A} O(i, \pi(i))$. As $|A| = m$, every orbit $O(i, \pi(i))$ must have cardinality four, and this holds if and only if $\pi$ is fixed-point-free. Consequently, the sets $A$ and $B$ must be disjoint. The proof now follows by fixing the bijection $\sigma: A \to B$, where $\sigma(i) \eqdef \pi(i)$ for all $i \in A$.
\end{proof}

Consider the orbits in $\mathbb{S}_3$ under the action of the $90^\circ$-rotation $r \in D_4$.

\begin{figure}[H]
    \centering
    \begin{tikzpicture}[scale=0.75]
	   \draw (0,-0.8) node[draw,circle, minimum size=0.5cm, inner sep=2pt] (a132) {$\mathtt{132}$};
	   \draw (-1.75,0) node[draw,circle, minimum size=0.5cm, inner sep=2pt] (a231) {$\mathtt{231}$};
	   \draw (0,0.8) node[draw,circle, minimum size=0.5cm, inner sep=2pt] (a213) {$\mathtt{213}$};
	   \draw (1.75,0) node[draw,circle, minimum size=0.5cm, inner sep=2pt] (a312) {$\mathtt{312}$};

	   \draw (5,0) node[draw,circle, minimum size=0.5cm, inner sep=2pt] (a123) {$\mathtt{123}$};
	   \draw (8,0) node[draw,circle, minimum size=0.5cm, inner sep=2pt] (a321) {$\mathtt{321}$};

	   \draw[-latex, line width=0.5pt] (a132) to [bend left] (a231);
	   \draw[-latex, line width=0.5pt] (a231) to [bend left] (a213);
	   \draw[-latex, line width=0.5pt] (a231) to [bend left] (a213);
	   \draw[-latex, line width=0.5pt] (a213) to [bend left] (a312);
	   \draw[-latex, line width=0.5pt] (a312) to [bend left] (a132);

	   \draw[-latex, line width=0.5pt] (a123) to [bend left] (a321);
	   \draw[-latex, line width=0.5pt] (a321) to [bend left] (a123);
    \end{tikzpicture}
    \caption{Orbits in $\mathbb{S}_3$ under the action $\langle r \rangle \acts \mathbb{S}_3$, where $r \in D_4$.}
    \label{fig:action_D4_S3}
\end{figure}
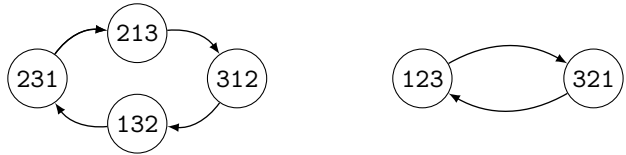
Referring to \Cref{fig:action_D4_S3} and \Cref{lemma:pc_D4}, yields the following useful fact regarding rotation-invariant permutations.

\begin{lemma} [$3$-profile of rotation-invariant permutations]
\label{lemma:big_rot_prof}
    If $\pi\in\Sn$ is rotation-invariant, then:
    \[
        \pc{123}{\pi}=\pc{321}{\pi}\text{ and }\pc{132}{\pi}=\pc{231}{\pi}=\pc{213}{\pi}=\pc{312}{\pi}
    \]
    In particular, $\pi$ is \threeb if and only if $\pc{123}{\pi}=\pc{132}{\pi}$.
\end{lemma}

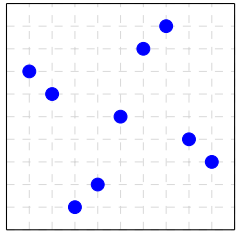
\begin{figure}[htb!]
    \centering
    \begin{tikzpicture}[scale=0.3]
        \draw[help lines, color=gray!30, dashed] (0,0) grid (10,10);
        \draw[thin](0,0)--(10,0) node[right]{};
        \draw[thin](0,0)--(0,10) node[above]{};
        \draw[thin](0,10)--(10,10) node[below]{};
        \draw[thin](10,0)--(10,10) node[left]{};

        \foreach \x/\y in {1/7,2/6,3/1,4/2,5/5,6/8,7/9,8/4,9/3} {
            \filldraw[blue] (\x,\y) circle (8pt) node[anchor=west]{};
        } 
    \end{tikzpicture}
    \caption{Plot of the \threeb permutation $\pi = \mathtt{761258943} \in \mathbb{S}_9$. By enumeration, the two shortest \threeb permutations are $\pi$ and its inverse (see also \cite{cooper2008symmetric}). Both are rotation-invariant.\protect\footnotemark }
    \label{fig:threebal1}
\end{figure}
\footnotetext{We note that for $n>9$, there appear to be many \threeb permutations which are \emph{not} rotation-invariant. These can be found, for instance, via a random-greedy search.}

Before we proceed to describe our construction, we note the arithmetic implications of \Cref{cor:div_conds} on any \threeb permutation.

\begin{lemma}[divisibility conditions for \threeb permutations]
    \label{lemma:div_cond_k_3}
    If $\pi\in\mathbb{S}_n$ is \threeb, then\\ $n\equiv 0$, $1$, $9$, $20$, $28$, or $29 \pmod {36}$.
\end{lemma} 

\subsubsection{A Rotation-Invariant Construction}
\Cref{lemma:big_rot_prof} suggests that we seek rotation-invariant permutations, while \Cref{prop:char_rotation_inv} provides a recipe for constructing such a permutation in terms of a bipartition and a bijection. By \Cref{lemma:div_cond_k_3}, there is a \threeb permutation in $\Sn$ only if $n\equiv 0$ or $1 \pmod {4}$. For our construction, we fix the bipartition $A \sqcup B$ where $A = \{ m + 1, \dots, 2m \}$ and $B=\{1, \dots, m\}$. The planar diagram of $\pi$ in $\RR^2$ looks as follows.
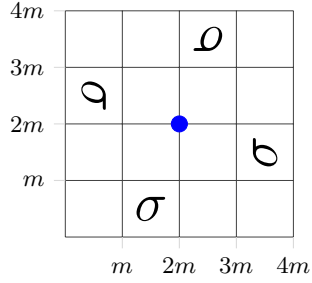
\begin{figure}[H]
    \centering
    \begin{tikzpicture}[scale=0.75]
        \draw[help lines, color=darkgray] (0,0) grid (4,4);

        \node[] at (0.5,2.5) {\huge{\rotatebox[origin=c]{-90}{$\sigma$}}};
		\node[] at (1.5,0.5) {\huge{\rotatebox[origin=c]{0}{$\sigma$}}};
		\node[] at (2.5,3.5) {\huge{\rotatebox[origin=c]{180}{$\sigma$}}};
		\node[] at (3.5,1.5) {\huge{\rotatebox[origin=c]{90}{$\sigma$}}};
        \filldraw[blue] (2,2) circle (4pt) {};

        \draw[color=gray!30] (1,0) -- (1,-0.2) node[below,color=black] {\small{$\phantom{1}m\phantom{1}$}};
        \draw[color=gray!30] (2,0) -- (2,-0.2) node[below,color=black] {\small{$2m$}};
        \draw[color=gray!30] (3,0) -- (3,-0.2) node[below,color=black] {\small{$3m$}};
        \draw[color=gray!30] (4,0) -- (4,-0.2) node[below,color=black] {\small{$4m$}};

        \draw[color=gray!30] (0,1) -- (-0.2,1) node[left,color=black] {\small{$m$}};
        \draw[color=gray!30] (0,2) -- (-0.2,2) node[left,color=black] {\small{$2m$}};
        \draw[color=gray!30] (0,3) -- (-0.2,3) node[left,color=black] {\small{$3m$}};
        \draw[color=gray!30] (0,4) -- (-0.2,4) node[left,color=black] {\small{$4m$}};

    \end{tikzpicture}
    \caption{Schematic plot of a rotation-invariant permutation $\pi \in \Sn$, where we fix the bipartition $\{1, \dots, m \} \sqcup \{ m + 1, \dots, 2m \}$ (see \Cref{prop:char_rotation_inv}). When $n$ is odd, we add the blue point at the centre. Here $\pi$ has the ``external structure'' of $\mathtt{3142} \in \mathbb{S}_4$. }
    \label{fig:perm_rot_inv}
\end{figure}

We now express the pattern counts of $\pi$ in terms of $\sigma$, as follows.

\begin{lemma}
    \label{lemma:pi_rot_inv_sigma_profile}
    Let $\sigma\in\mathbb{S}_m$, and let $\pi\in\mathbb{S}_n$ be obtained by rotation as in \Cref{fig:perm_rot_inv}, where $n=4m$. Then:
    \begin{align*}
        \pc{123}{\pi} & = 2\cdot\pc{123}{\sigma} + 2\cdot\pc{321}{\sigma} + 4m\cdot\pc{12}{\sigma}+2m\cdot\pc{21}{\sigma} \\
        \pc{132}{\pi} & = \pc{132}{\sigma} + \pc{231}{\sigma} + \pc{213}{\sigma} + \pc{312}{\sigma} + m^3 + m\cdot\pc{12}{\sigma} + 2m\cdot\pc{21}{\sigma}
    \end{align*}
    In particular, $\pi$ is \threeb if and only if:
    \begin{equation}
        \label{eq:sigma3bal}
        3\cdot \pc{123}{\sigma}+3\cdot\pc{321}{\sigma}+3m\cdot\pc{12}{\sigma}=\binom{m}{3}+m^3
    \end{equation}
\end{lemma}
\begin{proof}
    The expressions are obtained by case analysis. Any occurrence of a pattern $\tau \in \mathbb{S}_3$ in $\pi$ can be composed in three ways: either by taking all three points from the same ``block'', or by taking a pair from one block and a single point from another, or by picking one from each (see \Cref{fig:perm_rot_inv}). 
    
    For example, every ascending triplet in $\sigma$ and its $180^\circ$-rotation \rotatebox[origin=c]{180}{$\sigma$} contributes $1$ to $\pc{123}{\pi}$, hence the term $2\cdot \pc{123}{\sigma}$. Similarly, every choice of one element from each of $\{\rotatebox[origin=c]{0}{\ensuremath{\sigma}}, \rotatebox[origin=c]{180}{\ensuremath{\sigma}}, \rotatebox[origin=c]{90}{\ensuremath{\sigma}}\}$ forms a $\mathtt{132}$ pattern, hence the $m^3$ term in $\pc{132}{\pi}$. The other terms are obtained similarly. \Cref{eq:sigma3bal} now follows by rearranging and substituting the sum of $\sigma$'s $3$-profile by $\binom{m}{3}$.
\end{proof}
To construct an infinite \threeb family, it suffices to find permutations $\sigma\in\mathbb{S}_m$ that satisfy \Cref{eq:sigma3bal}. Initially, let us consider the following construction. Place three identical descending segments, each of length $\ell\geq 1$, in ascending order. As before, the patterns in $\sigma$ can be counted through case analysis. For example, an ascending pair is formed by choosing two of the three segments, and then one element from each. We obtain:
\[
    \pc{12}{\sigma} = 3 \ell^2, \qquad
    \pc{123}{\sigma} = \ell^3, \qquad
    \pc{321}{\sigma} = 3 \binom{\ell}{3}
\]
These values \emph{nearly} satisfy \Cref{eq:sigma3bal}. Indeed, both sides of the equation agree on the cubic and quadratic terms, and disagree \emph{only} on the linear terms. To achieve equality, we amend the construction slightly, by inserting two additional points ``in-between'' the existing ones (i.e., placing them at non-integer coordinates). For a parameter $\ell < r < 3\ell/2$ to be chosen below, and a small constant $0 < \varepsilon < 1$, the new coordinates are the following (see \Cref{fig:inc_segments_dec_with2}).
\[
    (x_1,y_1) \eqdef (r+2+\varepsilon,r+\ell+\varepsilon), \qquad
    (x_2,y_2) \eqdef (r+\ell+\varepsilon,r-\varepsilon)
\]
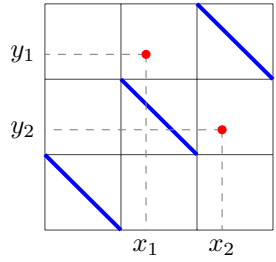
\begin{figure}[H]
    \centering
    \begin{tikzpicture}[scale=1]
        \draw[help lines, color=darkgray] (0,0) grid (3,3);

		\draw[color=blue, line width=1.5pt] (0,1) -- (1,0);
		\draw[color=blue, line width=1.5pt] (1,2) -- (2,1);
		\draw[color=blue, line width=1.5pt] (2,3) -- (3,2);

        \draw[color=gray, dashed] (1.33,2.33) -- (1.33,0) node[below, color=black] {$x_1$};
        \draw[color=gray, dashed] (2.33,1.33) -- (2.33,0) node[below, color=black] {$x_2$};
        
        \draw[color=gray, dashed] (1.33,2.33) -- (0,2.33) node[left, color=black] {$y_1$};
        \draw[color=gray, dashed] (2.33,1.33) -- (0,1.33) node[left, color=black] {$y_2$};

        \filldraw[red] (1.33,2.33) circle (1.5pt) {};
		\filldraw[red] (2.33,1.33) circle (1.5pt) {};

    \end{tikzpicture}
    \caption{Amending the basic construction of $\sigma\in\mathbb{S}_m$ by inserting two points, and fixing  $r=
    \lceil 4\ell / 3 \rceil$.}
    \label{fig:inc_segments_dec_with2}
\end{figure}
\begin{theorem}[$3$-balanced family]
    \label{thm:3balanced_construction}
    For every $n \ge 9$, there exists a \threeb permutation in $\Sn$ if and only if $n$ satisfies the divisibility conditions. That is, $n\equiv 0$, $1$, $9$, $20$, $28$, or $29 \pmod {36}$.
\end{theorem}

\begin{proof}
    Without loss of generality, we first consider the case $n \equiv 20 \pmod {36}$. By \Cref{lemma:pi_rot_inv_sigma_profile}, it suffices to calculate the pattern counts of the amended $\sigma\in \mathbb{S}_m$. Taking into account the two new points, and performing case analysis similarly to the above, we have:
    \begin{align*}
        \pc{12}{\sigma} &= 3 \ell^2 + (r - 1) + (2\ell-r) + (r + 2) + (2\ell-r) \\
        \pc{123}{\sigma} &= \ell^{3}+(r-\ell-1)\ell+(r-1)(2\ell-r)+(r-\ell+2)\ell+(r+2)(2\ell-r) \\
        \pc{321}{\sigma} &= 3 \binom{\ell}{3} + \binom{2\ell-(r-1)}{2} + \binom{r-\ell}{2} + \binom{2\ell-r-2}{2} + \binom{r-\ell}{2} + (3\ell - 2r - 1)
    \end{align*}
    
    \Cref{eq:sigma3bal} now simplifies to the condition $r=(4\ell+2)/3$. Therefore, writing $\ell=3t+1$ and $r=4t+2$, we obtain an infinite family of \threeb permutations, for every choice of $t\geq 2$. We have $|\sigma| = m = 3\ell+2=9t+5$ and therefore $|\pi|=n=36t+20$, so this yields a \threeb permutation for every $n > 56$ where $n \equiv 20 \pmod {36}$. The remaining residues (see \Cref{lemma:div_cond_k_3}) can be similarly handled, by amending $\sigma$ via a specifically chosen set of points. The details appear in \Cref{section:other_3bal}.
\end{proof}
\begin{rem}
    For $n$ that fails the divisibility conditions, this construction still produces \emph{nearly} balanced permutations. In particular, letting $\ell=3t$ or $\ell=3t+2$, and taking $r=\lfloor 4\ell/3 \rfloor +1$, the discrepancy in \Cref{eq:sigma3bal} is at most $\pm 2$.
\end{rem}

\section{Non-existence of \texorpdfstring{$k$}{k}-Balanced Permutations for \texorpdfstring{$k \ge 4$}{k >= 4}}
\label{section:4balanced_nonex}

In view of the results in \Cref{section:3balanced_construction}, one may seek $k$-balanced permutations for $k>3$. In this section we show that no such permutations exist. By the monotonicity proven in  \Cref{cor:k_bal_implies_lt_k_bal}, it suffices to show that there exist no $4$-balanced permutations. 

\subsection{Warmup: Ruling out \texorpdfstring{$k(n) \ge \log n + (2 + \varepsilon)\log \log n$}{k>=logn+(2+eps)loglogn}}

For a permutation $\pi \in \Sn$ to be $k$-balanced, it clearly must have \emph{at least} $|\mathbb{S}_k|$ $k$-tuples. By Stirling's approximation of the factorial this yields $k \lesssim e \sqrt{n}$. In fact, more is true: by \Cref{cor:k_bal_implies_lt_k_bal} the number of $r$-tuples in $\pi$ must be \emph{divisible} by $r!$, for all $r \le k$. This yields the following (see \cite{cooper2008symmetric} for further discussions of these divisibility conditions).

\begin{proposition}[ruling out $k \ge \log n + (2 + \varepsilon)\log \log n$]
    \label{prop:div_rule_bal}
    Let $k=k(n)$ be a function and let $\varepsilon > 0$ be a constant. If there exist $k(n)$-balanced permutations in $\Sn$, then for any sufficiently large $n$,
    \[ k(n) < \log n + (2 + \varepsilon) \log \log n \]
\end{proposition}
\begin{proof}
As usual, we denote by $\nu_2(t)$ the largest integer $s$ for which $2^s \mid t$.
Since $k!\mid \binom{n}{k}$,
    \[
        \nu_2(k!) \leq \nu_2 \left[ \binom{n}{k} \right] = \nu_2(n!) - \nu_2(k!) - \nu_2((n-k)!)
    \]
    
    It is a standard fact that $\nu_2(r!)=\sum_{i\geq 0} \lfloor r/2^i\rfloor$. The value of this sum is between
    $r$, and $r - \log r - \mathcal{O}(1)$. Consequently $2^k / k^2 = \mathcal{O}(n)$, which implies the proposition. 
\end{proof}

\begin{rem}
If we take $n=(k!)^2$, then $r! \mid \binom{n}{r}$ for all $r \in [k]$. It follows that divisibility alone does not imply $k(n) \le o(\log n/\log \log n)$, so that \Cref{prop:div_rule_bal} is tight up to $\mathcal{O}(\log \log n)$ factor.
\end{rem}

\subsection{Non-existence of \texorpdfstring{$4$}{4}-Balanced Permutations}

The following simple lemma provides a polynomial identity relating the $\{2,3,4\}$-profiles of any permutation. It is a direct corollary of this lemma that there exist no $4$-balanced permutations.

\begin{lemma}
    \label{lemma:relation_2_3_4_profile}
    Every permutation $\pi \in \Sn$ satisfies the following identity: 
    \begin{alignat*}{4}
        \left( \pc{12}{\pi} \right)^2 =\ &6 \cdot \pc{1234}{\pi} &&+ 4 \cdot \pc{1243}{\pi} &&+ 4 \cdot \pc{1324}{\pi} &&+ 2 \cdot \pc{1342}{\pi} \\
        +\ &2 \cdot \pc{1423}{\pi} &&+ 4 \cdot \pc{2134}{\pi} &&+ 4 \cdot \pc{2143}{\pi} &&+ 2 \cdot \pc{2314}{\pi} \\
        +\ &2 \cdot \pc{2413}{\pi} &&+ 2 \cdot \pc{3124}{\pi} &&+ 2 \cdot \pc{3142}{\pi} &&+ 2 \cdot \pc{3412}{\pi} \\
        +\ &6 \cdot \pc{123}{\pi} &&+ 2 \cdot \pc{132}{\pi} &&+ 2 \cdot \pc{213}{\pi} &&+ \pc{12}{\pi}
    \end{alignat*}
\end{lemma}
\begin{proof}
    Let us randomly sample independently and uniformly four indices from $[n]$. We consider the event that the sampled indices form two ascending pairs in $\pi$. By independence, 
    \begin{align*}
        \Pr_{i,j,k,l \sim [n]} \left[ i < j,\ \pi(i) < \pi(j),\ k < l,\ \pi(k) < \pi(l) \right] &= \Pr_{i,j \sim [n]} \left[ i < j,\ \pi(i) < \pi(j) \right]^2 \\
        &= \left( \frac{\pc{12}{\pi}}{n^2} \right)^2 
    \end{align*}
    
    The same event can also be computed as a weighted sum of pattern
    of $\le 4$ elements, by conditioning over the possible equalities between the sampled indices, and on their ordering. Fixing the set of indices in play and their order uniquely determines the patterns in $\pi$ that
    contribute to the event above. The computation then follows by total probability. To illustrate this analysis, we briefly analyse the first term in the above identity, leaving the full details to \Cref{appendix:relation_2_3_4_full_details}.
    
    Consider the case in which all the indices $i,j,k,l$ are distinct. This event happens with probability $\left(n(n-1)(n-2)(n-3) \right)/n^4$. Conditioned on this event, fix a total order on the indices. As the indices are sampled uniformly at random and there are no ties, each  order occurs with probability exactly $(1/4!)$. Under these two conditions, it only remains to enumerate over all patterns $\tau \in \mathbb{S}_4$ satisfying the original event, each of which contributes $\pc{\tau}{\pi}/\binom{n}{4}$. For example, if $i < j < k < l$, the corresponding patterns are:
    \[
        \pc{1234}{\pi} + \pc{1324}{\pi} + \pc{3412}{\pi} + \pc{2413}{\pi} + \pc{1423}{\pi} + \pc{2314}{\pi}
    \]
    and if $i < k < j < l$, we have:
    \[
        \pc{1324}{\pi} + \pc{1234}{\pi} + \pc{3142}{\pi} + \pc{2143}{\pi} + \pc{1243}{\pi} + \pc{2134}{\pi}
    \]
    The proof is concluded by applying total probability and collecting all terms (see  \Cref{appendix:relation_2_3_4_full_details}). 
\end{proof}

The non-existence of $4$-balanced permutations now follows directly.

\begin{theorem}
    \label{thm:no_4balanced}
    There are no $4$-balanced permutations.
\end{theorem}
\begin{proof}
    Arguing by contradiction, suppose that $\pi \in \Sn$ is a $4$-balanced permutation. By the monotonicity property of \Cref{cor:k_bal_implies_lt_k_bal}, there holds $\pc{\tau}{\pi} = \binom{n}{r}/r!$ for every $r \le 4$ and every $\tau \in \mathbb{S}_r$. Substituting into the identity of \Cref{lemma:relation_2_3_4_profile}, we obtain:
    \begin{equation}
        \label{eq:4bal_contradiction}
        \left[\frac{\binom{n}{2}}{2!}\right]^2 = \frac{36}{4!} \binom{n}{4} + \frac{10}{3!} \binom{n}{3} + \frac{1}{2!} \binom{n}{2} 
    \end{equation}
    and simplifying yields $n(n-1)(2n + 5) = 0$, a contradiction.
\end{proof}

\subsection{Comparison to the Quasirandomness Proof of Kr\`al' and Pikhurko}
\label{subsect:comparison_kp}

In this paper we examine profiles of permutations and the conditions under which they may be balanced. This is somewhat related to \emph{quasirandomness} of permutations, and the {\em asymptotic} convergence of the $k$-profile to uniform. Formally,

\begin{definition}[quasirandom permutations] 
    \label{defn:quasirandom}
    Let $\Pi = \{ \pi_n \}$ be an infinite family of permutations of non-decreasing order. We say that $\Pi$ is quasirandom if for every $k > 1$ and every $\tau \in \mathbb{S}_k$, we have:
    \[
        \frac{\pc{\tau}{\pi_n}}{\binom{n}{k}} \to \frac{1}{k!} \ \text{, as } n \to \infty 
    \]
\end{definition}

In an influential paper, \cite{kral2013quasirandom} Kr\`al' and Pikhurko proved a conjecture of Graham (see \cite{cooper2004quasirandom}), showing that every asymptotically $4$-balanced infinite family of permutations is quasirandom. Namely, if $\Pi$ has the property that $\pc{\tau}{\pi_n}/\binom{n}{4} \to \tfrac{1}{4!}$ (as $n \to \infty$) for every $\tau \in \mathbb{S}_4$, then $\Pi$ is quasirandom. 

{\em Permutons} or permutation limits are central to
the proof of \cite{kral2013quasirandom} (see also \cite{hoppen2013limits, hoppen2011limits}). A limit object in this framework is a {\em doubly stochastic} measure. I.e., a measure with \emph{uniform marginals} on the unit square $[0,1]^2$. Any such measure $\mu$ gives rise to a sampling process that produces permutations: just pick $k$ points uniformly at random from $\mu$, and consider the corresponding planar pattern. With probability $1$ these points define a permutation, since ties in any coordinate occur with zero probability. Kr\`al' and Pikhurko show that, up to sets of measure zero, the Lebesgue measure $\lambda$ is the one and only $4$-balanced, doubly stochastic measure on $[0,1]^2$.

\paragraph{Reinterpreting the proof of \cite{kral2013quasirandom}.} Consider two experiments. In the first, sample a point uniformly from the unit square, then sample two points \emph{independently} from the measure $\mu$. In the second experiment, we first sample a point uniformly from $[0,1]^2$, then \emph{one} point from $\mu$, and another point sampled uniformly from the unit square. In both experiments we consider the event that the first sampled point lies to the top-right of both subsequent points. The success probabilities of these experiments can be expressed in terms of $\mu$ and $\lambda$'s density functions $F,G: [0,1]^2 \to [0,1]$, respectively, on the bottom-left rectangles of the unit square. Concretely, for all $(a,b) \in [0,1]^2$,
\[
    F(a,b) \eqdef \mu([0,a] \times [0,b]), \text{ and } G(a,b) \eqdef \lambda([0,a] \times [0,b]) = ab
\]
From here the proof of \cite{kral2013quasirandom} proceeds by connecting between $F,G$ and the probabilities of the aforementioned events, and then computing these probabilities using the fact that $\mu$ is $4$-balanced. Rather than recount the proof, let us slightly delay the exposition and instead proceed directly to a discrete setting, where it is more convenient to provide full details.
\begin{figure}[H]
    \centering
    \begin{tikzpicture}[scale=1]
    
        \draw (0,0) -- (4,0) -- (4,4) -- (0,4) -- (0,0);

        \draw[pattern=hatch, pattern color=lightgray, opacity=0.5, hatch size=10pt, hatch angle=0] (0,0) rectangle +(3.25,3);

        \fill[red] (3.25,3) circle (3pt) {};
        \draw[darkgray] (3.25,3) circle (3pt) {};

        \node[] at (3.25,3.4) {
\small $(x, \pi(y))$};

        \fill[blue] (1,2.3) circle (3pt) {};
        \draw[darkgray] (1,2.3) circle (3pt) {};
        \node[] at (1,2.7) {
\small $(i, \pi(i))$};

        \definecolor{gold}{HTML}{FFC971}
        \fill[gold] (1.9,0.6) circle (3pt) {};
        \draw[darkgray] (1.9,0.6) circle (3pt) {};
        
       \node [inner sep=0pt,outer sep=0pt,rounded corners=0.2cm] (explabel) at (1.9,1.35) {
\small \begin{tabular}{@{}c@{}} \small {\fontsize{8}{12}\selectfont Exp A:} $(j, \pi(j))$ \\ \small {\fontsize{8}{12}\selectfont Exp B:} $(j, \pi(k))$ \end{tabular}};
       \node [fill=gold,opacity=0.5, fit=(explabel),rounded corners=.25cm,inner sep=2pt]  {};
       \node [draw=black,thin,fit=(explabel),rounded corners=.25cm,inner sep=2pt] {};
       \node [inner sep=0pt,outer sep=0pt,rounded corners=0.2cm] at (1.9,1.35) {
\small \begin{tabular}{@{}c@{}} \small {\fontsize{8}{12}\selectfont Exp A:} $(j, \pi(j))$ \\ \small {\fontsize{8}{12}\selectfont Exp B:} $(j, \pi(k))$ \end{tabular}};
       
    \end{tikzpicture}
    \caption{Two experiments. In the first, we sample a uniform grid point and two points from $\pi$, and consider the event that both $\pi$-points fall to the bottom-left of the initial point. The second experiment starts likewise, but the final point is sampled \emph{uniformly}. In the continuous setting one cannot distinguish the hybrid distribution from the original one, whereas in the discrete setting this can be done.}
\label{fig:experiment_kp}
\end{figure}
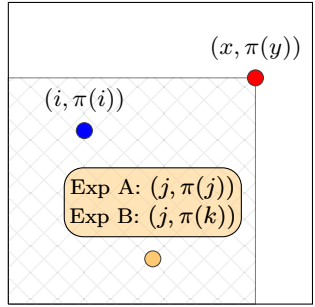

\paragraph{The discrete setting.} To recast \cite{kral2013quasirandom} in a discrete setting, let us replace the measure $\mu$ with a permutation $\pi \in \Sn$. In other words, rather than sample from $\mu$, we now sample points uniformly from the set $\{(i, \pi(i))\}_{i \in [n]}$. Instead of the two functions $F,G$ on the unit square defined above, we have two functions on the grid, $u,v: [n]^2 \to [0,1]$, which are defined by:
\[
    \forall x,y \in [n]^2:\ v(x,y) \eqdef \frac{xy}{n^2}, \text{ and } u(x,y) \eqdef \frac{|\{i \le x:\pi(i) \le y\}| }{n}
\]
Here $v$ is the bottom-left density function of the uniform doubly stochastic matrix $\frac{1}{n} \cdot \mathbbm{1} \otimes \mathbbm{1}$, and $u$ is the normalised number of points in the bottom-left rectangles of the permutation matrix associated with $\pi$ (i.e., the probability that a point chosen at random from $\pi$ falls in any such rectangle). As before, we would like to compute the probabilities of the events corresponding to the two experiments outlined above. To do so, observe that sampling a uniform point on the grid can be ``simulated'' by sampling two points independently and uniformly at random from $\{(i, \pi(i))\}_i$, keeping only the $x$-coordinate from the first point, and the $y$-coordinate from the second (discarding the two remaining coordinates). Since any permutation has ``uniform marginals'', this process yields a uniformly random point in the grid $[n]^2$. Consequently, we can now conduct the two experiments, and present their relation to the functions $u$ and $v$. By total probability, for the first experiment we have:
\begin{align*}
    \Pr_{i,j,x,y \sim [n]}[i,j \le x \land \pi(i),\pi(j) \le \pi(y) ] &= \sum_{s, t \in [n]} \Pr_{x,y \sim [n]}[x=s, \pi(y)=t] \cdot \Pr_{i,j \sim [n]}[i,j \le s \land \pi(i),\pi(j) \le t ] \\
    &= \frac{1}{n^2} \sum_{s,t \in [n]} \Pr_{i \sim [n]}[i \le s \land \pi(i) \le t]^2 \\
    &= \frac{1}{n^2} \sum_{s,t \in [n]} u(s,t)^2 = \frac{1}{n^2} \| u \| _2^2
\end{align*}
and for the second experiment: 
\begin{align*}
    \Pr_{i,j,k,x,y \sim [n]}[i,j \le x \land \pi(i),\pi(k) \le \pi(y) ] &= \sum_{s, t \in [n]} \Pr_{x,y \sim [n]}[x=s \land \pi(y)=t] \cdot \Pr_{i,j,k \sim [n]}[i,j \le s \land \pi(i),\pi(k) \le t ] \\
    &= \frac{1}{n^2} \sum_{s,t \in [n]} \Pr_{i \sim [n]}[i \le s \land \pi(i) \le t] \cdot \Pr_{j,k \sim [n]}[j \le s \land \pi(k) \le t] \\
    &= \frac{1}{n^2} \sum_{s,t \in [n]} u(s,t) \cdot v(s,t) = \frac{1}{n^2} \langle u, v \rangle
\end{align*}

Similarly to the proof of \Cref{lemma:relation_2_3_4_profile}, these two probabilities can \emph{also} be directly expressed in terms of weighted sums of permutation patterns in $\pi$, over at most $5$ points (in the continuous setting by applying Cauchy-Schwartz one can make do with only $4$-point patterns). However, this is where the proof from the continuous setting no longer carries over to the discrete case.

In the continuous setting, both events can be shown to occur with the same probability, which is precisely $\frac{1}{9}$. Then, since $\| G \|_2^2 = \frac{1}{9}$, it follows that $\langle F,G \rangle^2 = \| F \|_2^2 \| G \|_2^2$, and by Cauchy-Schwarz, $F = G$ (up to a set of measure zero), and this implies $\mu = \lambda$. In the discrete setting this does not hold. The primary difference being that we must \emph{also} consider the events in which ties occur, and unlike the continuous setting, these events have non-zero probabilities. Factoring in the possibility of ties and assuming that $\pi$ is a $5$-balanced permutation (and thus also $4$-balanced, see \Cref{cor:k_bal_implies_lt_k_bal}), we obtain the following identities (once again, by a computation similar to the proof of \Cref{lemma:relation_2_3_4_profile}):
\begin{align*}
\Pr_{i,j,x,y \sim [n]}[i,j \le x \land \pi(i),\pi(j) \le \pi(y) ] &= \frac{1}{9} + \frac{13}{36n} + \frac{7}{18n^2} + \frac{5}{36n^3} \\
\Pr_{i,j,k,x,y \sim [n]}[i,j \le x \land \pi(i),\pi(k) \le \pi(y) ] &= \frac{1}{9} + \frac{1}{3n} + \frac{13}{36n^2} + \frac{1}{6n^3} + \frac{1}{36n^4} = \frac{\| v \|_2^2}{n^2}    
\end{align*}

While the two probabilities agree on the leading term (which corresponds to the event of no ties), they \emph{disagree} on the remaining terms. Therefore we cannot apply Cauchy-Schwarz to argue that $u=v$ (which would have indeed yielded a contradiction, since the function $u$ associated with any permutation has precisely $n$ different values, whereas $v$ has $\Theta(n^2)$ different values and therefore does not correspond to any permutation), and must pursue a different proof. 

\section{The Minimal Distance from \texorpdfstring{$k$}{k}-Balanced for \texorpdfstring{$k \ge 4$}{k >= 4}}
\label{section:min_dist_kbal}

As we just saw, permutations cannot be $k$-balanced for any $k \ge 4$. But what is the \emph{smallest possible distance} (in, say, $\ell_\infty$-norm) between an \emph{attainable} profile and the uniform profile? Here is what we know:

\paragraph{Lower bound.} For any $k \ge 4$, we show that the minimal distance from $k$-balanced is at least $\Omega\left(n^{k-1} \right)$. The proof follows from a robust version of \Cref{lemma:relation_2_3_4_profile}, and is given in \Cref{subsect:lowerbound}.

\paragraph{Upper bound.} For $k=4$, we provide an explicit construction of an infinite family of permutations whose members attain a distance $\mathcal{O}\left(n^3\right)$ from uniform. The construction is based on modification of the well-known \ESZ permutation \cite{erdos1935combinatorial}, and is given in \Cref{subsect:upperbound}. Consequently, our bounds on the $4$-profile are asymptotically tight to within a constant factor. The remaining cases, where $k > 4$, are presently left open.

\paragraph{Concentration and anti-concentration} The asymptotic distribution of the $k$-profile is a well-researched topic \cite{even2020patterns, janson2013asymptotic, hofer2018central, bona2007copies}. We observe that these results imply that for any fixed $k \ge 2$, with probability $> 99\%$ the $k$-profile of a uniformly random $\pi \sim \Sn$ has distance $\Theta(n^{k-1/2})$ from balanced, as $n \to \infty$. Therefore, \emph{if} our lower bound for $k \ge 4$ is not tight for any $k$, then it is off by a multiplicative factor of $\mathcal{O}(n^{1/2})$ (see \Cref{subsubsect:random_perms} for a discussion).
\subsection{A Lower Bound on the Distance}
\label{subsect:lowerbound}

\begin{notation}[distance from uniform $k$-profile]
    Let $\pi\in\Sn$ be a permutation and let $1 \le k \le n$ be an integer. The distance of $\pi$ from the uniform $k$-profile, in $\ell_\infty$-norm, is denoted as follows:
    \[
        \delta_{\pi,k} \eqdef \max_{\tau\in\mathbb{S}_k} \left| \pc{\tau}{\pi} - \frac{\binom{n}{k}}{k!} \right|.
    \]
    We also denote the \emph{smallest} distance over \emph{all} $n$-element permutations by $\delta_k(n) \eqdef \min_{\pi \in \Sn} \delta_{\pi, k}$.
\end{notation}

\begin{lemma}[low-distance $k$-profile implies low-distance $(k-1)$-profile]
    \label{lemma:kclose_implies_lt_k_close}
    For every $\pi \in \Sn$ and $1 < k \le n$:
    \[
        \delta_{\pi,k-1} \leq \frac{k^2}{n-k+1} \delta_{\pi,k}
    \]
\end{lemma}
\begin{proof}
Pick some $\pi \in \Sn$ and $\tau\in\mathbb{S}_{k-1}$.
Slightly modifying the proof of Equation (\ref{eq:induce}) and \Cref{cor:k_bal_implies_lt_k_bal}, we write:
    \[
        (n-k+1) \cdot \pc{\tau}{\pi} = \sum_{\sigma\in\mathbb{S}_k} \pc{\tau}{\sigma} \cdot \pc{\sigma}{\pi} \leq \left(\frac{\binom{n}{k}}{k!} + \delta_{\pi,k} \right)\sum_{\sigma\in\mathbb{S}_k} \pc{\tau}{\sigma} = \left(\frac{\binom{n}{k}}{k!} + \delta_{\pi,k} \right)\cdot k^2
    \]
    Therefore,
    \[
        \pc{\tau}{\pi} \leq \frac{\binom{n}{k-1}}{(k-1)!} + \frac{k^2}{n-k+1}\cdot \delta_{\pi,k}
    \]
    The lower bound follows similarly.
\end{proof}

\begin{theorem} [lower bound on distance from $k$-balanced]
    \label{thm:lowerbound_distance}
    For every constant $k\geq 4$, there holds $\delta_k(n) = \Omega(n^{k-1})$.
\end{theorem}
\begin{proof}
In the proof of \Cref{lemma:relation_2_3_4_profile},
we derived \Cref{eq:4bal_contradiction} by
expressing $\pc{12}{\pi} \cdot \pc{12}{\pi}$ as a combination of pattern counts.
To extend of this proof to $k$-profiles, consider 
the product of $\pc{12}{\pi}$ and $\pc{(1,2,\ldots,k-2)}{\pi}$. 
Fix $k \geq 4$, and suppose toward contradiction that
there exists a sequence of positive integers $N_1 < N_2 < \dots$ and corresponding
permutations of orders $N_1, N_2,\ldots$
for which $\delta_k(N_t)/N_t^{k-1} =o_t(1)$. By
\Cref{lemma:kclose_implies_lt_k_close}, the same permutations also yield
$\delta_r(N_t)/N_t^{r-1} =o_t(1)$ for every 
$1 < r \leq k$.  In other words, every pattern in the $r$-profile of 
$\pi_N$ is $o(N^{r-1})$ away from $\binom{N}{r}/r!$.
    
As in the proof of \Cref{lemma:relation_2_3_4_profile}, we equate between two ways to express
the product of pattern counts in $\pi_N$. On the one hand:
    \begin{align}
        \label{eq:12k_first}
        \pc{12}{\pi_N} \cdot \pc{(1,\ldots,k-2)}{\pi_N} &= \left( \frac{\binom{N}{2}}{2!} \pm o(N) \right) \left( \frac{\binom{N}{k-2}}{(k-2)!} \pm o(N^{k-3}) \right) \nonumber \\
        &= \frac{1}{4(k-2)!^2}\cdot N^k + \frac{-k^2+5k-8}{8(k-2)!^2} \cdot N^{k-1} \pm o(N^{k-1})
    \end{align}

    On the other hand, we express the product as a sum of patterns of lengths $k$, $k-1$, and $k-2$, obtained from all possible ways in which the two patterns can be combined. 
    In this discussion it is helpful to think of a permutation as an axis-unaligned set of points in the grid.
    To account for the $k$-patterns that are generated, consider the insertion of an ascending pair into the permutation $(1,\ldots,k-2)$. There are $k-1$ possible $x$-coordinates at which we may insert the first point, and then $k$ for the second point. This counts every pair twice, so there are $k(k-1)/2$ options. The same is true of the $y$-coordinates, giving $\left(k(k-1)/2\right)^2$ patterns of length $k$.

    If a $(k-1)$-pattern is formed, necessarily one of the ascending pair's elements coincides with a points of $(1,\ldots,k-2)$, and the other does not. Suppose the former has coordinates $(i,i)$, for some $1\leq i\leq k-2$. To insert another element before it, such that an ascending pair is formed, we can freely insert a point among the permutation $(1,\ldots,i-1)$, and there are $i^2$ ways to choose its coordinates. Similarly, to insert an element to the top-right of $(i,i)$, there are $(k-1-i)^2$ options. Summing these squares over $1\leq i\leq k-2$ gives $2\cdot (k-2)(k-1)(2k-3)/6$. The remaining case, in which a $(k-2)$-pattern is formed, is negligible in this calculation. Indeed, there are at most $\mathcal{O}(N^{k-2})$ such patterns, each occurring a constant number of times. Overall, we have:
    \[
        \left(\frac{k(k-1)}{2}\right)^2 \left( \frac{\binom{N}{k}}{k!} \pm o(N^{k-1}) \right) + \frac{2(k-2)(k-1)(2k-3)}{6} \left( \frac{\binom{N}{k-1}}{(k-1)!} \pm o(N^{k-2}) \right) + \mathcal{O}(N^{k-2})
    \]
    and therefore:
    \begin{align}
        \label{eq:12k_second}
        \pc{12}{\pi_N} \cdot \pc{(1,\ldots,k-2)}{\pi_N} = \frac{1}{4(k-2)!^2} \cdot N^k + \frac{-3k^3 +22k^2 -59k+48}{24(k-1)!(k-2)!}\cdot N^{k-1} \pm o(N^{k-1})
    \end{align}
    
    \Cref{eq:12k_first} and \Cref{eq:12k_second} agree on their leading terms, but not on the second-order term for any $k \ge 4$. The contradiction follows by taking a sufficiently large $N$.
\end{proof}

\subsection{Matching Upper Bound on the Distance for \texorpdfstring{$k=4$}{k=4}}
\label{subsect:upperbound}
In this subsection we show that \Cref{thm:lowerbound_distance} is asymptotically tight when $k=4$. I.e., there exists an infinite family of permutations, the $4$-profiles of which are only $\mathcal{O}(n^{3})$ away from balanced. Our construction is a modification of the classical \emph{\ESZ permutation} \cite{erdos1935combinatorial}.

\begin{definition}[\ESZ permutations]\label{def:esz}
    \label{defn:erdos_szekeres}
    For integers $n,m \ge 1$ and $\theta>0$, let $\mathcal{P}(\theta)$ be the set of points in the $[n] \times [m]$ grid, rotated by an angle of $\theta$ about the origin. Let $\delta>0$ be the smallest angle such that some two points in $\mathcal{P}(\delta)$ reside on the same axis-parallel line. Pick some $0<\varepsilon<\delta$. The positive (resp.\ negative) \ESZ Permutation, denoted $\ES^+(n,m)$ (resp.\ $\ES^-(n,m)$) is the permutation associated with the point set $\mathcal{P}(\varepsilon)$ (resp.\ $\mathcal{P}(-\varepsilon)$). When $n=m$, we omit the second operand.
\end{definition}

\paragraph{Motivation.} In \cite{kral2013quasirandom} it was shown that for any $k\ge 4$, the \emph{unique} measure corresponding to a limit permutation with balanced $k$-profiles is the Lebesgue measure on the unit square. This suggests that in search of permutations with nearly balanced $k$-profiles, one may consider ``square-like'' families whose members locally resemble this measure. While the \ESZ permutations are natural candidates in this respect, their distance from uniform substantially exceeds the cubic bound.

\begin{proposition} [\ESZ is far from $4$-balanced]
    \label{prop:distance_es_basic}
    Let $n > 1$ be an integer and let $\pi = \ES^+(n)$ be the \ESZ permutation over $n^2$ elements. Then $\pc{3142}{\pi} = \binom{n+2}{4}^2$, and in particular:
    \[
         \delta_{\pi, 4} \ge \frac{1}{144} n^{7} + \Omega(n^6)  
    \]
\end{proposition}
\begin{proof}
    Let $n > 1$, let $\pi = \ES^+(n)$ and let $\tau = \mathtt{3142}$. Let $A$ be the following set of four points, $A = \{(x_1, y_1), (x_2, y_2), (x_3, y_3), (x_4, y_4)\} \subseteq ([n] \times [n])$, whose $x$-coordinates are in weakly ascending order, $x_1 \le x_2 \le x_3 \le x_4$.
    
    In order for $A$ to form an instance of $\tau$ in $\pi$, its $y$-coordinates must weakly agree with the ordering of $\tau$, i.e., $y_2 \le y_4 \le y_1 \le y_3$. As the points of $\pi$ correspond to the rotation of the $[n] \times [n]$ grid by a small positive angle, any pair of points on a horizontal line becomes an ascending pair, and any pair on a vertical line becomes a descending pair. Therefore, since $\tau(2) < \tau(3)$, we necessarily have $x_2 < x_3$ (they cannot lie on a vertical line), and since $\tau(1) > \tau(4)$ we necessarily have $y_1 > y_4$ (they cannot lie on a horizontal line). These two conditions imply that $|A|=4$, since any two points in $A$ disagree on some coordinate. In fact, these conditions prevent \emph{all} ascending pairs of $\tau$ from lying on vertical line, and \emph{all} descending pairs from lying on a horizontal line, and are therefore \emph{sufficient} in order for $A$ to induce a copy of $\tau$. Consequently,
    \begin{align*}
        \pc{\tau}{\pi} &= \Big| \big\{x_1 \le x_2 < x_3 \le x_4,\  y_2 \le y_4 < y_1 \le y_3\ :\ x_1,\dots,x_4 \in [n],\ y_1,\dots,y_4 \in [n] \big\} \Big| \\
        &= \Big| \big\{x_1 \le x_2 < x_3 \le x_4\ :\ x_1,\dots,x_4 \in [n] \big\} \times \big\{ y_2 \le y_4 < y_1 \le y_3\ :\ y_1,\dots,y_4 \in [n] \big\} \Big| \\
        &= \Big| \big\{x_1 \le x_2 < x_3 \le x_4\ :\ x_1,\dots,x_4 \in [n] \big\} \Big|^2
    \end{align*}

    In the latter case there are four possible choices to consider: a set of four distinct points; a triple with either ($x_1$ and $x_2$) or ($x_3$ and $x_4$) identified; or, a pair with both the aforementioned identifications. Therefore, 
    \begin{align*}
        \pc{\tau}{\pi} &= \Big| \big\{x_1 \le x_2 < x_3 \le x_4\ :\ x_1,\dots,x_4 \in [n] \big\} \Big|^2 = \left( \binom{n}{4} + 2\binom{n}{3} + \binom{n}{2} \right)^2 = \binom{n+2}{4}^2 \qedhere
    \end{align*}
\end{proof}

\begin{rem}
    Apart from $\tau = \mathtt{3142}$, there is only one other pattern $\sigma = \mathtt{2413}$ in $\mathbb{S}_4$, for which the distance is $\Theta(n^7)$. The computation for this pattern proceeds identically to the proof of \Cref{prop:distance_es_basic}, the only difference being the set of strict inequalities imposed. All remaining entries of the $4$-profile of $\ES^+(n)$ are indeed within $\mathcal{O}(n^6)$ of uniform.
\end{rem}

\subsubsection{A Modification of \ESZ}
\label{subsect:es_mod}

We next define a modification of the \ESZ permutation.

\begin{definition} [two-sided \ESZ]
        With $n,m$ and $\varepsilon$ as in \cref{def:esz}, the two-sided \ESZ Permutation, denoted $\ES^\pm(n,m)$, is the permutation associated with the points $\mathcal{P}(\varepsilon) \sqcup \mathcal{P}(-\varepsilon)$.
\end{definition}

The $4$-profile of this family of permutations is optimally balanced, up to a multiplicative constant.

\begin{theorem} [$4$-profile of two-sided \ESZ]
    \label{thm:prof_mod_es}
    Let $n \ge 1$ and let $\pi = \ES^{\pm}(n) \in \mathbb{S}_{2n^2}$ be the two-sided \ESZ permutation. Then, 
    \[
        \delta_{\pi, 4} = \frac{2n^6}{9} + \mathcal{O}(n^5)
    \]
\end{theorem}
\begin{proof}
    Consider two copies of the $[n] \times [n]$ grid; one blue and one red. We slightly rotate the blue grid counterclockwise around the origin, so that any two points on a horizontal line become an ascending pair, and any two points on a vertical line become a descending pair. The red copy of the grid is likewise rotated clockwise, so that the above rules are reversed. These two rotations also enforce particular alignment (ascending or descending) to every bi-coloured pair of points on an axis-aligned line.
    
    We turn to count the occurrences of any $\tau \in \mathbb{S}_4$ induced by the two grids. Let us first describe a counting process that applies to all patterns in $\mathbb{S}_4$ other than $\mathtt{2413}$ and $\mathtt{3142}$. In this process, our calculations are carried out on the integer grid, keeping in mind the rotations of the blue and red copies. So, consider some other pattern, say $\tau = \mathtt{1243}$.
    Any instance of $\tau$ is determined by first fixing two ascending grid points and their colours, and naming them $\perm{2}$ and $\perm{3}$, respectively. It remains to fix $\perm{1}$ and $\perm{4}$. We observe that \emph{given} a choice of the first two points, the feasible regions for either of the remaining points are defined by disjoint \emph{rectangles}, whose corners are determined by the boundaries of the grid and by the coordinates of $\perm{2}$ and $\perm{3}$. Thus, for a fixed choice of the two points, the number of $\tau$-instances is simply the product of the volumes of both rectangles, where {\em volume} means the number of grid points (of either colour) within said rectangle.
    
    \begin{figure}[H]
        \centering
        \begin{tikzpicture}[scale=1]

            \draw[help lines, color=darkgray, opacity=0.8] (0,0) grid (5,5);
            
            \fill[color=white] (0.01,0.01) rectangle (1.99,1.99);
            \fill[color=gray, opacity=0.1] (0,0) rectangle (2,2);
            \draw [pattern=north east lines, distance=10pt, pattern color=darkgray, opacity=0.4] (0,0) rectangle (2,2);
            \node[] at (1,1) {1};
            \draw [dashed, color=red, line width=0.3mm] (2,2.03) -- (0,2.03);
            \draw [dashed, color=blue, line width=0.3mm] (2.04,1.97) -- (0,1.97);
            \draw [color=blue, line width=0.3mm] (2.03,2) -- (2.03,0);
            \draw [color=red, line width=0.3mm] (1.97,2) -- (1.97,0);

            \fill[color=white] (2.01,3.01) rectangle (3.99,4.99);
            \fill[color=gray, opacity=0.1] (2,3) rectangle (4,5);
            \draw [pattern=north east lines, distance=10pt, pattern color=darkgray, opacity=0.4] (2,3) rectangle (4,5);
            \node[] at (3,4) {4};
            \draw [color=red, line width=0.3mm] (2,3.03) -- (4,3.03);
            \draw [color=blue, line width=0.3mm] (2.03,2.97) -- (4.09,2.97);
            
            \draw [color=red, line width=0.3mm] (1.97,3) -- (1.97,5);
            \draw [dashed, color=blue, line width=0.3mm] (2.03,3) -- (2.03,5);

            \draw [dashed, color=red, line width=0.3mm] (4.03,3) -- (4.03,5);
            \draw [color=blue, line width=0.3mm] (3.97,2.92) -- (3.97,5);

            \fill[red] (4,3) circle (3pt) {};
            \draw[darkgray] (4,3) circle (3pt) {};
            \fill[red] (2,2) circle (3pt) {};
            \draw[darkgray] (2,2) circle (3pt) {};
            \node[] at (1.8,2.3) {2};
            \node[] at (4.2,3.3) {3};

        \end{tikzpicture}
        \caption{Counting the pattern $\tau = \mathtt{1243}$ in $\ES^\pm(n)$. Fixing the points $\perm{2}$ and $\perm{3}$ determines the feasible rectangles for $\perm{1}$ and $\perm{4}$. Coloured solid lines indicate that points of that colour may be taken at the boundary. Dashed lines indicate that they may not. }
        \label{fig:count_1243_es_pm}
    \end{figure}
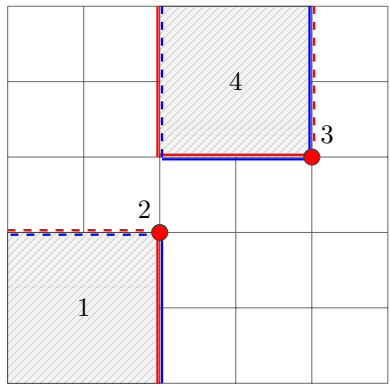

    Crucially, we remark that the colours of the two fixed points determine the ``tie-breaks'' in computing the volumes; for instance, if $\perm{2}$ is a red point, then the bottom-left rectangle from which $\perm{1}$ is sampled may include either colour on a vertical line with $\perm{2}$, but cannot include red or blue points on a horizontal line with $\perm{2}$. In more detail: say $R$ and $B$ are a red and a blue grid point, respectively, to the left of $\perm{2}$. Then, the clockwise rotation of the red grid puts the rotated $R$ higher than the rotated $\perm{2}$, and the counterclockwise rotation of the blue grid puts the rotated $B$ higher than the clockwise rotated $\perm{2}$.

    Consequently, for any $\tau \in \mathbb{S}_4 \setminus \left\{ \mathtt{2413}, \mathtt{3142} \right\}$, the number of occurrences of $\tau$ is characterised by a sum (over all two combinations of two fixed points) of the products of volumes of rectangles determined by these points. Rather than directly compute these sums, we take a shortcut: note that such an expression is a polynomial in $n$. Indeed, the volumes of the rectangles are clearly polynomials in the coordinates of their corners, and in $n$, and the sum is taken over all choices of the two fixed points. The remaining two patterns, $\mathtt{2413}$ and $\mathtt{3142}$ are inverses of one another, and by construction $\ES^\pm(n)$ is an involution. Thus $\pc{\sigma}{\ES^\pm(n)} = \pc{\sigma^{-1}}{\ES^\pm(n)}$ for \emph{any} permutation $\sigma$, and in particular, $\pc{2413}{\ES^\pm(n)} = \pc{3142}{\ES^\pm(n)}$. Since the sum of all $4$-profiles is $\binom{n}{4}$, a polynomial in $n$, the remaining two pattern-counts are therefore \emph{also} polynomials in $n$.
    
    To conclude, every entry in the $4$-profile of $\ES^\pm(n)$ is a polynomial in $n$, of degree at most $8$ (there are only $\Theta(n^8)$ four-tuples). The proof now follows by directly computing $\pc{\tau}{\ES^{\pm}(n)}$ for all $n \in \{1, \dots, 9\}$, and for every $\tau \in \mathbb{S}_4$, and applying Lagrange interpolation over these points.
\end{proof}

\subsubsection{Profiles and Distance of Random Permutations}
\label{subsubsect:random_perms}

A simple probabilistic argument shows that for every \emph{fixed} $k \ge 2$ and large $n$, the $k$-profile of almost every permutation in $\Sn$ is $\left(\binom{n}{k}/k! \pm o(n^k)\right) \mathbbm{1}$. In this discussion we are interested in exactly how close to balanced the $k$-profile of a typical (random) permutation is, and in particular, whether this distance attains, or nearly attains, our lower bound of \Cref{thm:lowerbound_distance}.

So fix some $k \ge 2$ and consider a pattern $\tau \in \mathbb{S}_k$. Associated with $\tau$ is the random variable $X_\tau \eqdef \pc{\tau}{\pi}$ where $\pi$ is uniformly sampled from $\Sn$. Clearly, $\E[X_\tau] = \binom{n}{k}/{k!}$. The distribution of $X_\tau$, its moments, and even the pairwise joint distributions of patterns have received considerable attention (e.g., \cite{even2020patterns, janson2013asymptotic, hofer2018central, bona2007copies}). It is known in particular that  $X_\tau$ satisfies a central limit theorem. Concretely, there exists a constant $\sigma_\tau > 0$ such that as $n \to \infty$,
\[
    \sqrt{n} \left(\frac{ X_\tau }{\binom{n}{k}} - \frac{1}{k!} \right) \xrightarrow[\hphantom{space} ]{d} N(0, \sigma_\tau)
\]

This CLT implies asymptotic concentration and anti-concentration of $k$-profiles.

\begin{proposition}[concentration and anti-concentration of $k$-profile]
    Let $k \ge 2$ be a constant. Then, for any $\alpha > 0$ we have:
    \[
        \Pr_{\pi \sim \Sn} \left[ \delta_{\pi, k} \ge \alpha \cdot n^{k-1/2} \right] = 2 \cdot \Phi\left(-\frac{\alpha}{k!}\right) \pm o(1)\ \ \ \ \ \ 
    \]
    and conversely (by union over $\mathbb{S}_k$),
    \[
        \Pr_{\pi \sim \Sn} \left[ \delta_{\pi, k} \le \alpha \cdot n^{k-1/2} \right] \ge 1 - 2 k! \cdot \Phi\left(-\frac{\alpha}{k!}\right) \pm o(1)
    \]
\end{proposition}

In particular, this implies that $\delta_{\pi, k} \in \left(\frac{1}{100 k!},\ 2 k!\right)n^{k-1/2}$ with probability $> 99\%$ as $n \to \infty$.

\section{An Asymptotic Relation Between Profiles and Permutations}
\label{section:profiles_to_perms}
So far we have considered $k$-profiles of order-$n$ permutations with $k > 1$ \textit{fixed}, and $n \to \infty$.
But such profiles are interesting also when $k=k(n)$ grows with $n$.
For example, as observed at the start of \Cref{section:4balanced_nonex}, when $k \gtrsim e \sqrt{n}$,
at least some order-$k$ permutations must be missing, i.e., the support of the $k$-profile is necessarily incomplete.
Also, in the extreme case where $k=n$, the $k$-profile is a singleton. Our main discovery in this section 
is the following:

\paragraph{Profiles determine points.} In the range 
$n\ge k(n) \ge \Omega(\sqrt{n} \log n)$, the $k$-profile of $\pi\in\Sn$ 
reveals a lot about $\pi$.
Explicitly, we prove that there exists a set $\mathcal{D} \subset [n]^2$ of 
$\widetilde{\Omega}(k^4/n^2)$ points
(consisting of four symmetric regions, of widths roughly $k^2/n$),
such any two permutations in $\Sn$ with the same $k$-profile,
\textit{must} agree on their restriction to $\mathcal{D}$ 
(see \Cref{fig:profiles_to_points}).
In the extreme case where $k=n$, our Theorem is close to tight,
as the set $\mathcal{D}$ nearly covers the entire grid $[n] \times [n]$ (up to a logarithmic factor).

\begin{figure}[H]
    \centering
    \begin{tikzpicture}[scale=1]
        \draw (0,0) -- (4,0) -- (4,4) -- (0,4) -- (0,0);
        \draw[color=lightgray, opacity=0.5] (0,2) -- (4,2);
        \draw[color=lightgray, opacity=0.5] (2,0) -- (2,4);
        
        \def\BL{(0,0.9) -- (0.7,0.9) arc (90:0:0.2cm) -- (0.9,0) -- (0, 0) -- (0, 0.9)}
        \draw[pattern=hatch, pattern color=red, opacity=0.8] \BL;

        \def\TL{(0,3.1) -- (0.7,3.1) arc (-90:0:0.2cm) -- (0.9,4) -- (0, 4) -- (0, 3.1)}
        \draw[pattern=hatch, pattern color=red, opacity=0.8] \TL;

        \def\TR{(4,3.1) -- (3.3,3.1) arc (-90:-180:0.2cm) -- (3.1,4) -- (4, 4) -- (4, 3.1)}
        \draw[pattern=hatch, pattern color=red, opacity=0.8] \TR;

        \def\BR{(4,0.9) -- (3.3,0.9) arc (90:180:0.2cm) -- (3.1,0) -- (4, 0) -- (4, 0.9)}
        \draw[pattern=hatch, pattern color=red, opacity=0.8] \BR;

        \draw[color=blue, dashed] (0.9,4) -- (0.9,0);
        \draw[color=blue, dashed] (3.1,4) -- (3.1,0);
        \draw[color=blue, dashed] (0, 0.9) -- (4, 0.9);
        \draw[color=blue, dashed] (0, 3.1) -- (4, 3.1);

        \draw[<->] (1.2,3.1) -- (1.2,4) node[right, midway]{\small $\widetilde{\Omega}\left(\sfrac{k^2}{n}\right)$};
        \draw[->, gray] (-0.2,-0.2) -- (-0.2, 4.3) node[above, darkgray] {$x$};
        \draw[->, gray] (-0.2,-0.2) -- (4.3, -0.2) node[right, darkgray] {$y$};

        \draw[color=gray!30] (0,0) -- (0,-0.3) node[below,color=darkgray] {\small{$1$}};
        \draw[color=gray!30] (4,0) -- (4,-0.3) node[below,color=darkgray] {\small{$n$}};

        \draw[color=gray!30] (0,0) -- (-0.3,0) node[left,color=darkgray] {\small{$1$}};
        \draw[color=gray!30] (0,4) -- (-0.3, 4) node[left,color=darkgray] {\small{$n$}};
        
    \end{tikzpicture}
    \caption{The $k$-profile fully determines the restriction of any $\pi \in \Sn$ with that profile, to the red region.}
	\label{fig:profiles_to_points}
\end{figure}
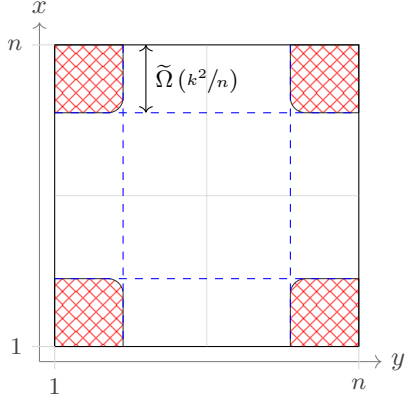

The main result of this section is \Cref{thm:prof_to_points}. Its proof goes
as follows: We define the {\em evaluation} $p(\pi)$ of a bivariate polynomial $p$ over a permutation $\pi$. Then we show that if $\deg(p)<k$, then this real number $p(\pi)$ is uniquely defined by the $k$-profile of the permutation $\pi$. We subsequently use standard tools from approximation theory to construct a family of polynomials of degree $<k$, which allow us to uncover the points in $\mathcal{D}$.

\subsection{\texorpdfstring{$k$}{k}-Profiles Determine the Evaluation of Degree \texorpdfstring{$< k$}{<k} Polynomials}
\label{subsect:k_prof_determines_deg_k_polys}

Here is the main notion that we use in this section:

\begin{notation}[evaluation of polynomial on permutation]
    \label{notn:eval_poly_perm}
    Let $p \in \RR[x,y]$ be a real bivariate polynomial and let $\pi \in \Sn$ be a permutation. The evaluation of $p$ on $\pi$ is denoted:
    \[
        p(\pi) \eqdef \sum_{i=1}^n p(i, \pi(i))
    \]    
\end{notation}  

With this notation we show:

\begin{proposition}[$k$-profile determines $(\deg < k)$-evaluations]
    \label{prop:k_prof_to_poly}
    Let $p \in \RR[x,y]$ be a bivariate polynomial with $\deg(p) < k$ for some integer $k > 1$. Also, let $\pi \in \Sn$ be a permutation of order $n\ge k$. Then $p(\pi)$ can be efficiently computed, given $\pi$'s $k$-profile. 
\end{proposition}
\begin{proof}
    For any two integers $1 < t < r < k$, consider the following event, in which all indices are sampled uniformly at random and independently from $\{1, \dots, n\}$:
    \begin{align*}
        \Pr_{\substack{x_1, \dots, x_t\\ y_1, \dots, y_{r-t} \\ i}} \hspace{-0.2cm}\big[ x_1, \dots, x_t \le i,y_1, \dots, y_{r-t} \le \pi(i) \big] &= \sum_{z=1}^n \Pr_{\substack{x_1, \dots, x_t \\ y_1, \dots, y_{r-t}}} \hspace{-0.1cm}\big[ x_1, \dots, x_t \le z, y_1, \dots, y_{r-t} \le \pi(z) \big] \Pr_{i} [i=z] \\
        &= \frac{1}{n} \sum_{z=1}^n \Pr_{x_1, \dots, x_t} \big[ x_1, \dots, x_t \le z] \Pr_{y_1, \dots, y_{r-t}} \big[ y_1, \dots, y_{r-t} \le \pi(z) \big] \\
        &= \frac{1}{n} \sum_{z=1}^n z^r \cdot \pi(z)^{t-r} = \frac{1}{n} (x^r y^{t-r})(\pi)
    \end{align*}
    where the first equality follows from the law of total probability, and the latter ones by independence. Conversely, by conditioning on the possible equalities between the sampled indices, the same event can be expressed as a weighted sum of permutation patterns over $(\le k)$-points (see proof of \Cref{lemma:relation_2_3_4_profile}). By \Cref{prop:down_ind_k_prof}, fixing the $k$-profile determines all $r$-profiles, where $r < k$. This proves the proposition for $p(x,y)=x^t y^{r-t}$, and the proof now follows, since these monomials span all bivariate polynomials of degree $<k$.\footnote{Since permutations are bijections from $[n] \to [n]$, every coordinate appears exactly once. Therefore, every $\pi \in \Sn$ must agree on the evaluations of all polynomials in which either $x$ or $y$ do not appear, regardless of the total degree.}   
\end{proof}

\begin{rem}
    \Cref{prop:k_prof_to_poly} can be extended by analysing a different sets of events. For example, for any $r < k$ and $\tau \in \mathbb{S}_r$, we could consider the event in which we sample $r$ permutation points, and condition on their relative ordering so that they form an instance of $\tau$ in $\pi$. Then, using the remaining budget of at most $k-r$ points, we could sample from their marginals. Such events determine the evaluations of many more polynomials (albeit, on a modified and weighted pointset).
\end{rem}

\subsection{Determining Points using Approximate Indicators}
\label{subsect:approx_indicators_and_points}

Here is our method for ``reading the bit'' at position $(x,y)$. We construct
to this end a \emph{low-degree} polynomial that is a good 
pointwise approximator of the indicator $\mathbbm{1}_{(x,y)}: [n]^2 \to \{0,1\}$.
If the polynomial has degree $<k$ and the approximation error is small, then by
evaluating it, we can determine the value of the corresponding bit. This means that either
\emph{every} permutation $\pi$ with a given $k$-profile must contain this point,
or \emph{none} of them do, and the evaluation of the polynomial will reveal this. 

For notational convenience, in what follows we consider 
(as in \Cref{section:3balanced_construction}) the action 
$\langle r \rangle \acts [1,n]^2$ of the $90^\circ$-rotation. We denote by 
$O(a,b) = \{ (a,b), (b, n+1-a), (n+1-a, n+1-b), (n+1-b, a)\}$ the $r$-orbit of 
$(a,b) \in [n]^2$. Similarly, for a set $\mathcal{D} \subset [n]^2$ we denote $O(\mathcal{D}) \eqdef \cup_{(a,b) \in \mathcal{D}} O(a,b)$ (i.e., the $r$-orbit of $\mathcal{D}$). The following fact is well-known and easy to verify (e.g., \cite{minsky1969perceptron}). 

\begin{lemma}[symmetrisation]
    \label{lemma:mp_sym}
    Let $p : \{0,1\}^n \to \RR$ be a real multilinear polynomial, and let $f: \{0, 1, \dots, n\} \to \RR$ be the function:
    \[
        \forall k \in \{0,1, \dots n\}: f(k) = \E_{|x| \sim k}[p(x)]
    \]
    where the expectation is taken with respect to the uniform distribution over all $x \in \{0,1\}^n$ of Hamming weight $k$. Then, $f$ can be written as a real polynomial in $k$ of degree at most $\deg(p)$.
\end{lemma}

\begin{lemma}[approximate degree of symmetric boolean functions \cite{paturi1992degree}]
    \label{lemma:apx_deg_sym}
    Let $f: \{0,1\}^n \to \{0,1\}$ be a symmetric Boolean function, and let:
    \[
        \Gamma(f) = \min_k \{ |2k - n + 1| : f_{k} \ne f_{k+1} \}
    \]
    where $f_k$ is the value of $f$ on inputs of Hamming weight $k$. Then, there exists a multilinear polynomial $g \in \RR[x_1, \dots, x_n]$ such that $\forall x \in \{0,1\}^n: |g(x) - f(x)| \le 1/3$, and furthermore:
    \[
        \deg(g) \le A \cdot \sqrt{n \left( n - \Gamma(f) \right) }
    \]
    where $A > 0$ is a universal constant. 
\end{lemma}

\begin{lemma}[one-sided approximation of $\mathbbm{1}_{(a,b)}$]
\label{lemma:approx_ind}
Let $u,v$ and $n$ be integers such that $1 \le u,v < n/2$. Then, for any $(a,b) \in O(u,v)$, there exists a polynomial $\widetilde{\mathbbm{1}}_{(a,b)} \in \RR[x,y]$ such that:
\[
    \deg \left(\widetilde{\mathbbm{1}}_{(a,b)}\right) \le C \left( \sqrt{n(2u+1)} + \sqrt{n(2v+1)}\right) \log n
\]
where $C > 0$ is an absolute constant, and:
\[
    \forall (x,y) \in [n]^2:\ \widetilde{\mathbbm{1}}_{(a,b)}(x,y) \in \begin{cases}
        [1, \infty) & x=a \land y=b \\
        [0, \frac{1}{2n}] & x \ne a \lor y \ne b
    \end{cases}
\]
\end{lemma}
\begin{proof}
    For any $t \in [n]$, let $H_t: \{0,1\}^n \to \{0,1\}$ be the symmetric Boolean function $H_t(x) \eqdef \mathbbm{1}\{ |x| = t \}$, where $|x|$ is the Hamming weight of $x \in \{0,1\}^n$. By construction, $\Gamma(H_t) \in |2t - n \pm 1|$. An application of \Cref{lemma:apx_deg_sym} to $H_t$, gives a real multilinear polynomial $G_t \in \RR[x_1, \dots, x_n]$ such that $\forall x \in \{0,1\}^n: | H_t(x) - G_t(x) | \le 1/3$, and whose degree is bounded by $A (n(n - \Gamma(H_t))^{1/2}$ where $A > 0$ is an absolute constant (independent of $n,t$). Consider $f_t: \{0, 1, \dots, n\} \to \RR$, the symmetrisation of $G_t$. By \Cref{lemma:mp_sym}, $f_t$ is a \emph{univariate polynomial} of degree at most $\deg(G_t)$. Since $H_t$ is constant over all inputs of the same Hamming weight, and $G_t$ approximates $H_t$ pointwise to error at most $1/3$, we have:
    \[
        \forall x \in [n]:\ \big| f_t(x) - \mathbbm{1}_t (x) \big| \le \frac{1}{3} \implies \left(f_t(x) + \frac{1}{3}\right) \in \begin{cases}
            [1, \frac{4}{3}] & x = t \\
            [0, \frac{2}{3}] & x \ne t
        \end{cases}
    \]
    To conclude the proof, let $(a,b) \in O(u,v)$ where $u,v < n/2$, and consider the following bivariate polynomial:
    \[
        \widetilde{\mathbbm{1}}_{(a,b)}(x,y) = \left[ \left(f_a(x) + \frac{1}{3}\right)\left(f_b(y) + \frac{1}{3}\right) \right]^{\log_{\sfrac{3}{2}}2n} \in \RR[x,y]
    \]
    Taking the products and powers of the aforementioned bounds on $f_t$, it follows that:
    \[
        \forall (x,y) \in [n]^2:\ \widetilde{\mathbbm{1}}_{(a,b)}(x,y) \in \begin{cases}
            [1, \infty) & x=a \land y=b \\
            [0, (\frac{2}{3})^{\log_{\sfrac{3}{2}}2n} = \frac{1}{2n}] & x \ne a \lor y \ne b 
        \end{cases}
    \]
    Lastly, by construction, the total degree of $\widetilde{\mathbbm{1}}_{(a,b)}$ is bounded by $A (\sqrt{n(2u+1)} + \sqrt{n(2v+1)}) \log_{\sfrac{3}{2}} (2n)$, and the proof now follows for an appropriate choice of $C$.
\end{proof}

\begin{theorem}[$k$-profiles determine points]
    \label{thm:prof_to_points}
    Let $n \ge k > 1$ and let:
    \[
        \mathcal{D} = \Big\{ (a,b) \in \left[ n/2 \right]^2\ :\ C \left( \sqrt{n(2a+1)} + \sqrt{n(2b+1)}\right) \log n < k \Big\} \subset [n]^2
    \]
    where $C$ is the constant of \Cref{lemma:approx_ind}.
    Then the $k$-profile of an order-$n$ permutation $\pi \in \Sn$ uniquely determines the restriction of $\pi$ to $O(\mathcal{D})$. 
\end{theorem}
\begin{proof}
    Let $(a,b) \in O(\mathcal{D})$ and let $\widetilde{\mathbbm{1}}_{(a,b)}$ be the one-sided approximation given by \Cref{lemma:approx_ind}. By construction, for every permutation $\pi \in \Sn$ we have $\widetilde{\mathbbm{1}}_{(a,b)}(\pi) \ge 1$ iff $(a,b) \in \{(i, \pi(i)) : i \in [n]\}$, and $\mathbbm{1}_{(a,b)}(\pi) \le (1/2n) \cdot n \le 1/2$ otherwise. So, this evaluation determines the presence or absence of $(a,b)$. From \Cref{lemma:approx_ind}, it follows that $\deg(\widetilde{\mathbbm{1}}_{(a,b)}) < k$, and thus (by \Cref{prop:k_prof_to_poly}), all permutations with a given $k$-profile must \emph{agree} on this polynomial, and on the coordinate $(a,b)$.
\end{proof}

\section{Discussion}
\label{sect:discussion}

In this paper we consider the existence of $k$-balanced permutations.
For $k \le 3$ we show that such permutations exist whenever $n$ satisfies the necessary divisibility conditions,
and for $k \ge 4$, we show that no such permutations exist.
Moreover, we prove that the $k$-profile of any $n$-element permutation must have an entry
which is $\Omega_n(n^{k-1})$ away from uniform, whenever $k \ge 4$.
This gives rise to several interesting open questions.

\paragraph{Is the lower bound tight?} 
Recall that for $k=4$ we provide an explicit construction of an infinite family (see \Cref{subsect:es_mod})
in which every pattern in $\mathbb{S}_4$ appears within additive distance of $\Theta(n^{3})$ from uniform, i.e., matching the lower bound of \Cref{thm:lowerbound_distance}.
Conversely, we note (see \Cref{subsubsect:random_perms}) that \textit{all} entries in the $k$-profile of a uniformly random permutation in $\mathbb{S}_n$ 
are, with probability $>99\%$ (for large enough $n$), within distance $\Theta(n^{k - 1/2})$ from uniform. In this view we ask
what is the true behaviour for $k>4$. Specifically,
does our lower bound remain tight, or does the true bound change to $\Omega(n^{k-1/2})$, as attained by the majority of permutations?

\paragraph{How many $k$-patterns can appear the right number of times?}
We have ruled out the possibility
that \textit{every} entry in the $4$-profile equals $\binom{n}{4}/4!$.
However, for any \textit{fixed} pattern $\tau \in \mathbb{S}_4$, we are able
to construct a bespoke infinite family of permutations, in whose members $\tau$ appears exactly $\binom{n}{4}/4!$ times
(these constructions are quite intricate, and are not included in this paper).
So, we ask: how \textit{many} entries in the $k$-profile of an $n$-element permutation may be precisely $\binom{n}{k}/k!$, simultaneously? 

\paragraph{What is the maximal dimension of a $k$-balanced subspace?} 
It makes sense to ask the same question with regards to linear subspaces.
That is, what is the maximal dimension of a subspace $V_k \le \RR^{\mathbb{S}_k}$ such that there exist infinitely many permutations $\pi \in \Sn$ for which:
\[
    \forall v = (\alpha_{\tau})_{\tau \in \mathbb{S}_k} \in V_k:\ \frac{\binom{n}{k}}{k!} \langle v, \mathbbm{1}_{\mathbb{S}_k} \rangle = \sum_{\tau \in \mathbb{S}_k} \alpha_\tau \pc{\tau}{\pi}
\]

In other words, unlike the previous question, here we allow any basis for $V_k$, not necessarily only the coordinate vectors.
Clearly $\langle \mathbbm{1}_{\mathbb{S}_k} \rangle \in V_k$, for any $k$. Also, since $3$-balanced permutations exist, then by \Cref{prop:down_ind_k_prof} there
are $3!=6$ linearly independent combinations in $\mathbb{S}_4$ that hold true, and $\dim(V_4) \ge 6$ ($\mathbbm{1}_{\mathbb{S}_k}$ resides in their span). In general, we ask: what is the maximal dimension of $V_k$, for $k \ge 4$?

\paragraph{How many permutations are $3$-balanced?} $2$- and $3$-balanced permutations \textit{exist} for every admissible value of $n$ (see \Cref{section:3balanced_construction}). In fact, they never appear ``alone'': as they are closed under the action of $D_4$, their entire orbit is also balanced and so there must at least be \textit{two} balanced permutations, whenever one exists (no permutation is identical to its reflection about either axis). Therefore, we ask: what is the \textit{exact count}, or even \textit{asymptotic growth rate}, of $3$-balanced permutations (restricted only to the admissible $n$)? We remark that for $2$-balanced permutations these answers are already known (see \cite[A316775]{oeis} and \cite[A000140]{oeis}). However, interestingly, for $k=3$ we presently only know that at $n=9$ there are exactly two \threeb permutations (see \Cref{fig:threebal1}).

\bibliography{balanced_perms}
\bibliographystyle{alpha}

\begin{appendices}
\crefalias{section}{appendix}
\pagebreak
\section{3-Balanced Constructions for all Remainders}
\label{section:other_3bal}

In the proof of \Cref{thm:3balanced_construction} we amended $\sigma$ by inserting two points, yielding \threeb permutations for every $n > 56$ with $n \equiv 20 \pmod {36}$. The same strategy applies to the other residues as well, as we now describe. In the following discussion $\ell = 3t+1$, and $0 < \varepsilon < 1$ is a small constant. 

To prove the correctness of these constructions, we observe that all newly inserted coordinates are affine transformations of $t$. Consequently, each $2$-pattern (resp.\ $3$-pattern) count in $\sigma$ is a quadratic (resp.\ cubic) polynomial in $t$. It follows that \Cref{eq:sigma3bal} posits the vanishing of a cubic polynomial. This we can verify by checking only four distinct values of $t$. This is an alternative to the calculation presented in \Cref{thm:3balanced_construction}.

\subsection{Constructions for even \texorpdfstring{$n$}{n}}
\label{subsect:constructions_even_n}

In \Cref{thm:3balanced_construction}, we amended $\sigma\in\mathbb{S}_m$ by inserting two additional points:

\addstackgap[7pt]{\begin{tabular}{p{1.5cm} p{5.5cm} p{5.5cm}}
& $(x_1,y_1) \eqdef (r+2+\varepsilon,\ r+\ell+\varepsilon)$ & $(x_2,y_2) \eqdef (r+\ell+\varepsilon,\ r-\varepsilon)$
\end{tabular}}

\noindent and showed that for any $t\geq 2$, one can set $r=4t+2$ to satisfy \Cref{eq:sigma3bal}. The size of the resulting \threeb permutation $\pi\in\mathbb{S}_n$ obtained by rotation is $n=4m=36t+20$, that is, $n \equiv 20 \pmod {36}$

\paragraph{The case $n \equiv 28 \pmod {36}$.} Insert two more points to $\sigma$, so that $m=3\ell+4$:

\addstackgap[7pt]{\begin{tabular}{p{1.5cm} p{5.5cm} p{5.5cm}}
& $(x_3,y_3) \eqdef (\varepsilon,\ \ell+\varepsilon)$ & $(x_4,y_4) \eqdef (1+\varepsilon,\ \ell-\varepsilon)$
\end{tabular}}

\paragraph{The case $n \equiv 0 \pmod {36}$. } Insert two more points, in addition to the above four:

\addstackgap[7pt]{\begin{tabular}{p{1.5cm} p{5.5cm} p{5.5cm}}
& $(x_5,y_5) \eqdef (\ell+2+\varepsilon,\ \varepsilon)$ & $(x_6,y_6) \eqdef (\ell+1+\varepsilon,\ 5+\varepsilon)$
\end{tabular}}

\subsection{Constructions for odd \texorpdfstring{$n$}{n}}
\label{subsect:constructions_odd_n}

For $t \ge 4$, inserting the following points (and a point at the centre of $\pi$) yields \threeb permutations.

\paragraph{The case $n \equiv 29 \pmod {36}$.} Insert four points to $\sigma$, so that $m=3\ell+4$:

\addstackgap[7pt]{\begin{tabular}{p{1.5cm} p{5.5cm} p{5.5cm}}
\vspace{0.07in} & $(x_1,y_1) \eqdef (-5 + \varepsilon,\ 1 + \varepsilon)$ & $(x_2,y_2) \eqdef (-3 + \varepsilon,\ t - 2 + \varepsilon)$ \\
& $(x_3,y_3) \eqdef (-2 + \varepsilon,\ 7t + 4 + \varepsilon)$ & $(x_4,y_4) \eqdef (\quad\quad\ \varepsilon,\ 7t + 3 + \varepsilon)$ 
\end{tabular}}

\paragraph{The case $n \equiv 1 \pmod {36}$.} Insert two more points to $\sigma$, so that now $m=3\ell+6$:

\addstackgap[7pt]{\begin{tabular}{p{1.5cm} p{5.5cm} p{5.5cm}}
& $(x_5,y_5) \eqdef (-4 + \varepsilon,\ 3t - 1 + \varepsilon)$ & $(x_6,y_6) \eqdef (2t - 2 + \varepsilon,\ 5t + 1 + \varepsilon)$
\end{tabular}}

\paragraph{The case $n \equiv 9 \pmod {36}$.} Insert two last points to $\sigma$, in addition to the previous six. So, $m=3\ell+8$:

\addstackgap[7pt]{\begin{tabular}{p{1.5cm} p{5.5cm} p{5.5cm}}
& $(x_7,y_7) \eqdef (-1 + \varepsilon,\ 3t - 2 + \varepsilon)$ & $(x_8,y_8) \eqdef (4t + 1 + \varepsilon,\ t - 1 + \varepsilon)$
\end{tabular}}

\subsection{Small cases not covered by our construction}

For completeness, we provide a list of $3$-balanced permutations in \Cref{table:other_3bal}, for these values of $n$ that were not covered by the aforementioned constructions. This is because \Cref{thm:3balanced_construction}, \Cref{subsect:constructions_even_n} and \Cref{subsect:constructions_odd_n} yield \threeb permutations for every residue modulo $36$, starting only from some minimal value of $t$. 

\begin{table}[H]
    \centering
    \small
    \setlength{\tabcolsep}{10pt}
    \renewcommand{\arraystretch}{1.1}
    \begin{tabular}{ ||c p{14.3cm}|| } \hline
    \multicolumn{1}{||c}{\textbf{$n$}} & \multicolumn{1}{c||}{\textbf{$3$-Balanced Permutation in $\mathbb{S}_n$}}  \\
    \hline \hline

    $9$ & (3, 4, 9, 8, 5, 2, 1, 6, 7)  \\
    
    \hline
    
    $20$ & (8, 3, 19, 16, 4, 11, 9, 20, 14, 6, 15, 7, 1, 12, 10, 17, 5, 2, 18, 13)  \\
    
    \hline
    
    $28$ & (1, 7, 18, 22, 19, 26, 9, 23, 28, 14, 6, 10, 24, 12, 13, 5, 17, 3, 20, 16, 8, 2, 4, 25, 15, 27, 21, 11)  \\

    \hline

    $29$ & (9, 3, 6, 17, 26, 14, 29, 13, 11, 22, 19, 23, 28, 25, 20, 18, 8, 12, 4, 2, 5, 1, 7, 10, 15, 16, 27, 24, 21) \\
    
    \hline
    
    $36$ & (30, 12, 27, 4, 25, 15, 3, 10, 32, 5, 35, 33, 19, 17, 13, 24, 16, 1, 20, 9, 36, 21, 28, 8, 23, 14, 34, 18, 7, 26, 6, 22, 2, 29, 31, 11) \\

    \hline
    
    $37$ & (16, 15, 1, 35, 9, 13, 29, 30, 33, 32, 12, 3, 34, 21, 14, 18, 8, 19, 10, 20, 36, 22, 31, 17, 24, 28, 11, 5, 2, 6, 4, 7, 37, 26, 25, 27, 23) \\

    \hline
    
    $45$ & (42, 14, 15, 9, 36, 5, 44, 8, 26, 33, 17, 39, 21, 29, 11, 1, 25, 32, 20, 45, 35, 12, 27, 7, 3, 23, 22, 38, 24, 16, 31, 40, 19, 4, 10, 37, 41, 34, 2, 28, 6, 43, 30, 13, 18) \\
    
    \hline
    
    $56$ & (7, 8, 38, 50, 42, 25, 27, 45, 10, 22, 15, 46, 17, 56, 21, 36, 12, 24, 54, 2, 48, 39, 14, 16, 30, 51, 55, 41, 5, 26, 31, 35, 49, 4, 43, 34, 13, 3, 1, 37, 29, 28, 40, 32, 11, 23, 18, 19, 20, 44, 6, 52, 33, 53, 9, 47) \\
    
    \hline
    
    $64$ & (14, 8, 23, 48, 60, 38, 55, 50, 44, 26, 4, 62, 24, 53, 31, 35, 1, 47, 9, 28, 16, 33, 20, 40, 7, 63, 6, 43, 19, 29, 11, 13, 52, 54, 36, 46, 22, 59, 2, 58, 25, 45, 32, 49, 37, 56, 18, 64, 30, 34, 12, 41, 3, 61, 39, 21, 15, 10, 27, 5, 17, 42, 57, 51)  \\
    
    \hline
    
    $65$ & (59, 33, 5, 4, 55, 26, 44, 57, 37, 19, 7, 47, 20, 34, 9, 51, 27, 22, 50, 16, 12, 64, 36, 10, 56, 45, 40, 65, 39, 43, 54, 14, 28, 13, 15, 24, 30, 62, 32, 46, 63, 35, 41, 6, 53, 3, 49, 38, 11, 29, 1, 60, 2, 23, 18, 17, 8, 31, 52, 61, 48, 21, 42, 25, 58)  \\
    
    \hline
    
    $72$ & (11, 21, 16, 58, 45, 55, 3, 49, 9, 65, 31, 44, 69, 7, 60, 66, 2, 71, 42, 37, 35, 34, 24, 40, 29, 39, 54, 51, 8, 47, 25, 36, 56, 72, 38, 6, 13, 41, 48, 30, 4, 12, 46, 28, 62, 14, 63, 67, 10, 53, 32, 19, 43, 52, 22, 50, 1, 18, 5, 70, 64, 17, 57, 61, 23, 26, 15, 27, 68, 33, 59, 20) \\
    
    \hline
    
    $73$ & (42, 68, 6, 13, 4, 70, 36, 38, 2, 29, 47, 24, 64, 51, 32, 48, 41, 54, 19, 20, 71, 60, 7, 53, 26, 67, 69, 37, 33, 18, 27, 44, 3, 30, 55, 72, 52, 49, 45, 16, 25, 14, 50, 12, 31, 5, 35, 61, 65, 46, 21, 9, 56, 10, 66, 8, 59, 43, 39, 1, 15, 22, 34, 23, 62, 73, 57, 11, 17, 63, 58, 28, 40) \\

    \hline
    
    $81$ & (44, 31, 58, 10, 77, 49, 80, 79, 1, 52, 34, 15, 5, 25, 63, 42, 9, 47, 23, 38, 32, 73, 35, 4, 27, 48, 46, 69, 22, 68, 41, 2, 30, 59, 81, 66, 65, 40, 26, 53, 74, 29, 21, 28, 11, 54, 6, 19, 18, 55, 70, 75, 56, 43, 60, 13, 57, 14, 61, 37, 76, 64, 36, 3, 17, 8, 51, 67, 78, 72, 7, 33, 20, 39, 62, 12, 24, 16, 50, 71, 45) \\
    
    \hline

    $101$ & (84, 58, 37, 76, 54, 38, 51, 20, 85, 83, 36, 22, 97, 7, 23, 39, 90, 6, 67, 88, 66, 11, 70, 93, 96, 4, 16, 2, 53, 34, 24, 3, 31, 48, 91, 13, 94, 74, 64, 14, 44, 21, 49, 46, 79, 18, 61, 43, 95, 17, 78, 56, 52, 69, 75, 57, 72, 41, 81, 32, 101, 73, 5, 98, 15, 42, 62, 89, 86, 87, 25, 68, 30, 50, 82, 33, 8, 60, 71, 45, 55, 27, 47, 63, 29, 1, 35, 9, 10, 65, 80, 19, 59, 92, 40, 12, 77, 26, 99, 28, 100) \\
    
    \hline

    $109$ & (64, 39, 93, 24, 97, 65, 7, 54, 101, 52, 49, 62, 88, 35, 9, 81, 4, 8, 89, 13, 75, 87, 36, 41, 80, 71, 28, 3, 55, 73, 1, 78, 99, 104, 82, 14, 22, 57, 18, 63, 67, 29, 20, 109, 91, 31, 43, 100, 32, 106, 15, 95, 69, 16, 76, 79, 45, 53, 85, 103, 90, 92, 61, 27, 50, 23, 86, 66, 38, 26, 60, 34, 70, 25, 5, 105, 37, 46, 84, 83, 44, 6, 72, 10, 12, 51, 94, 68, 96, 74, 17, 59, 56, 47, 2, 30, 102, 11, 48, 40, 19, 21, 42, 98, 33, 77, 58, 108, 107) \\
    
    \hline

    $117$ & (100, 99, 90, 107, 48, 67, 45, 88, 61, 50, 4, 21, 79, 89, 75, 23, 53, 1, 2, 56, 106, 26, 102, 87, 36, 96, 72, 3, 14, 8, 24, 47, 42, 60, 44, 93, 38, 81, 13, 69, 52, 85, 15, 83, 111, 27, 86, 113, 40, 108, 6, 77, 101, 55, 64, 98, 9, 34, 59, 84, 109, 20, 54, 63, 17, 41, 112, 10, 78, 5, 32, 91, 7, 35, 103, 33, 66, 49, 105, 37, 80, 25, 74, 58, 76, 71, 94, 110, 104, 115, 46, 22, 82, 31, 16, 92, 12, 62, 116, 117, 65, 95, 43, 29, 39, 97, 114, 68, 57, 30, 73, 51, 70, 11, 28, 19, 18) \\
    
    \hline

    $137$ & (29, 66, 131, 65, 27, 44, 3, 119, 117, 61, 101, 47, 28, 79, 92, 18, 49, 122, 8, 31, 9, 48, 81, 103, 76, 104, 133, 125, 137, 40, 118, 42, 38, 26, 24, 87, 11, 105, 68, 108, 95, 106, 41, 132, 75, 15, 126, 116, 121, 74, 36, 83, 80, 60, 52, 71, 23, 53, 14, 84, 128, 25, 45, 50, 134, 136, 56, 99, 69, 39, 82, 2, 4, 88, 93, 113, 10, 54, 124, 85, 115, 67, 86, 78, 58, 55, 102, 64, 17, 22, 12, 123, 63, 6, 97, 32, 43, 30, 70, 33, 127, 51, 114, 112, 100, 96, 20, 98, 1, 13, 5, 34, 62, 35, 57, 90, 129, 107, 130, 16, 89, 120, 46, 59, 110, 91, 37, 77, 21, 19, 135, 94, 111, 73, 7, 72, 109) \\
    
    \hline

    $145$ & (97, 26, 116, 136, 81, 11, 55, 32, 118, 4, 140, 50, 92, 103, 124, 69, 62, 23, 45, 44, 83, 15, 128, 67, 109, 144, 99, 9, 48, 3, 105, 138, 66, 75, 40, 94, 25, 59, 46, 111, 31, 95, 14, 126, 127, 107, 27, 117, 1, 134, 42, 36, 72, 13, 139, 57, 90, 60, 108, 88, 76, 129, 21, 78, 5, 113, 122, 64, 130, 61, 34, 93, 73, 53, 112, 85, 16, 82, 24, 33, 141, 68, 125, 17, 70, 58, 38, 86, 56, 89, 7, 133, 74, 110, 104, 12, 145, 29, 119, 39, 19, 20, 132, 51, 115, 35, 100, 87, 121, 52, 106, 71, 80, 8, 41, 143, 98, 137, 47, 2, 37, 79, 18, 131, 63, 102, 101, 123, 84, 77, 22, 43, 54, 96, 6, 142, 28, 114, 91, 135, 65, 10, 30, 120, 49) \\
    
    \hline

    $153$ & (142, 69, 119, 121, 148, 5, 118, 57, 105, 80, 130, 1, 84, 83, 28, 47, 38, 50, 76, 44, 109, 61, 25, 11, 131, 115, 95, 139, 31, 55, 125, 89, 4, 73, 3, 7, 66, 137, 26, 96, 48, 102, 56, 134, 21, 87, 138, 113, 9, 136, 101, 42, 51, 75, 124, 111, 146, 40, 27, 86, 132, 90, 82, 62, 32, 117, 46, 60, 152, 13, 14, 63, 120, 10, 100, 135, 77, 19, 54, 144, 34, 91, 140, 141, 2, 94, 108, 37, 122, 92, 72, 64, 22, 68, 127, 114, 8, 43, 30, 79, 103, 112, 53, 18, 145, 41, 16, 67, 133, 20, 98, 52, 106, 58, 128, 17, 88, 147, 151, 81, 150, 65, 29, 99, 123, 15, 59, 39, 23, 143, 129, 93, 45, 110, 78, 104, 116, 107, 126, 71, 70, 153, 24, 74, 49, 97, 36, 149, 6, 33, 35, 85, 12) \\
    
    \hline
    
    \end{tabular}
    
    \caption{$3$-balanced permutations for all values of $n$ not covered by our constructions.}
    \label{table:other_3bal}
\end{table}

\section{Computation of \texorpdfstring{\Cref{lemma:relation_2_3_4_profile}}{Lemma 4.2}}
\label{appendix:relation_2_3_4_full_details}
In the proof of \Cref{lemma:relation_2_3_4_profile} we fix $\pi \in \Sn$ and consider the event:
\[
    \Pr_{i,j,k,l \sim [n]} \left[ i < j,\ \pi(i) < \pi(j),\ k < l,\ \pi(k) < \pi(l) \right]
\]

The proof of \Cref{lemma:relation_2_3_4_profile} proceeds by showing that, by conditioning on the possible equalities between the sampled indices and on their total order, the aforementioned event can be expressed as a polynomial involving pattern-counts in $\pi$, of lengths at most $4$. The complete details of this computation are presented in \Cref{table:details_2_3_4}. The first column of the table corresponds to the partition over the indices, where two indices are identical if and only if they reside in the same set. The second column corresponds to a fixed linear order over the indices at play, and the third column lists the contributing patterns, conditioned over the first two events. The total probability is then computed by summing over the probabilities of each row. Every row containing a partition of cardinality $k$, a linear order, and a set of patterns $S \subset \mathbb{S}_k$, indicates an event occurring with probability:
\[
    \frac{\prod_{i=1}^k (n-i+1) }{n^4} \cdot \frac{1}{k!} \cdot \sum_{\tau \in S} \frac{\pc{\tau}{\pi}}{\binom{n}{k}}
\]

\begin{table}[H]
\centering
\setlength{\tabcolsep}{12pt}
\renewcommand{\arraystretch}{1.1}
\begin{tabular}{ ||l l l|| } \hline
\multicolumn{1}{||c}{\textbf{Index Partition}} & \multicolumn{1}{c}{\textbf{Linear Order}} &  \multicolumn{1}{c||}{\textbf{Contributing Patterns}} \\
\hline \hline
\multirow{6}{7em}{$\{i\}, \{j\}, \{k\}, \{l\}$ } & $i < j < k < l$ & $\mathtt{1234}$, $\mathtt{1324}$, $\mathtt{3412}$, $\mathtt{2413}$, $\mathtt{1423}$, $\mathtt{2314}$ \\ 
& $i < k < j < l$ & $\mathtt{1324}$, $\mathtt{1234}$, $\mathtt{3142}$, $\mathtt{2143}$, $\mathtt{1243}$, $\mathtt{2134}$\\ 
& $i < k < l < j$ & $\mathtt{1342}$, $\mathtt{1243}$, $\mathtt{3124}$, $\mathtt{2134}$, $\mathtt{1234}$, $\mathtt{2143}$ \\ 
& $k < i < j < l$ & $\mathtt{3124}$, $\mathtt{2134}$, $\mathtt{1342}$, $\mathtt{1243}$, $\mathtt{2143}$, $\mathtt{1234}$ \\ 
& $k < i < l < j$ & $\mathtt{3142}$, $\mathtt{2143}$, $\mathtt{1324}$, $\mathtt{1234}$, $\mathtt{2134}$, $\mathtt{1243}$ \\ 
& $k < l < i < j$ & $\mathtt{3412}$, $\mathtt{2413}$, $\mathtt{1234}$, $\mathtt{1324}$, $\mathtt{2314}$, $\mathtt{1423}$ \\
\hline
\multirow{2}{7em}{$\{i,k\}, \{j\}, \{l\}$ } & $i < j < l$ & $\mathtt{123}$, $\mathtt{132}$ \\
& $i < l < j$ & $\mathtt{123}$, $\mathtt{132}$ \\
\hline
\multirow{2}{7em}{$\{j,l\}, \{i\}, \{k\}$ } & $i < k < j$ & $\mathtt{123}$, $\mathtt{213}$ \\
& $k < i < j$ & $\mathtt{123}$, $\mathtt{213}$ \\
\hline
\multirow{1}{7em}{$\{i,l\}, \{j\}, \{k\}$} & $k < i < j$ & $\mathtt{123}$ \\
\hline
\multirow{1}{7em}
{$\{j,k\}, \{i\}, \{l\}$} & $i < j < l$ & $\mathtt{123}$ \\
\hline
\multirow{1}{7em}
{$\{i,k\}, \{j,l\}$} & $i < j$ & $\mathtt{12}$ \\
\hline
\end{tabular}
\caption{The terms corresponding to the computation of the total probability of the event analysed in the proof of \Cref{lemma:relation_2_3_4_profile}. The partitions and orderings which do not contribute are omitted. }
\label{table:details_2_3_4}
\end{table}

\end{appendices}

\end{document}